\documentclass{amsart}

\usepackage[all,hyperref]{bi-discrete}
\usepackage{amsmath}
\usepackage{todonotes}
\usepackage[backend=biber,maxbibnames=5,maxalphanames=5,style=alphabetic,sorting=nyt,bibencoding=utf8,giveninits,url=false,isbn=false]{biblatex}
\allowdisplaybreaks

\newcommand{\phix}{\ensuremath{{\phi(\cdot)}}}

\newcommand{\tildeA}{\widetilde{A}}
\newcommand{\tp}{\widetilde{\phi}}

\numberwithin{equation}{section}
\numberwithin{theorem}{section}
\begin{document}

\author[A.\ Cianchi]{Andrea Cianchi}
\address{Dipartimento di Matematica e Informatica \lq\lq U. Dini"\\
Universit\`a di Firenze\\
Viale Morgagni 67/a\\
50134 Firenze\\
Italy} 
\email{andrea.cianchi@unifi.it}

\author[L.\ Diening]{Lars Diening}
\address{Fakult\"at für Mathematik, University Bielefeld,
Universit\"atsstrasse 25, 33615 Bielefeld, Germany}
\email{lars.diening@uni-bielefeld.de}

\subjclass[2000]{46E35, 46E30}
%35J25, 35J60, 35B65.

\keywords{Musielak-Orlicz spaces, Sobolev inequalities, generalized Young functions, Riesz potentials}

\title{Sobolev embeddings in Musielak-Orlicz spaces
%Musielak-Orlicz-Sobolev embeddings
%Unifying Orlicz and variable-exponent Sobolev inequalities: a sharp Musielak-Orlicz-Sobolev embedding
}

\begin{abstract}
 An embedding theorem for Sobolev spaces built upon general Musielak-Orlicz norms is offered. These norms are defined in terms of generalized Young functions which also depend on the $x$ variable. Under minimal conditions on the latter dependence, a Sobolev conjugate is associated with any function of this type. 
Such a conjugate is sharp, in the sense that, for each fixed $x$, it agrees with the sharp Sobolev conjugate in classical Orlicz spaces. Both Sobolev inequalities in the whole $\RRn$   and   Sobolev-Poincar\'e inequalities in domains are established. Compact Sobolev embeddings are also presented. In particular,  optimal embeddings for standard Orlicz-Sobolev spaces,   variable exponent Sobolev spaces, and double-phase Sobolev spaces are recovered and complemented in borderline cases. A key tool, of independent interest, in our approach is a new weak type inequality for Riesz potentials in Musielak-Orlicz spaces involving a sharp fractional-order Sobolev conjugate.
\end{abstract}
%\thanks{The work of Lars Diening was supported by the Deutsche Forschungsge- meinschaft (DFG, German Research Foundation) – SFB 1283/2 2021 – 317210226.}
\maketitle

\section{Introduction}
Among the diverse extensions of the classical Lebesgue spaces $L^p(\Omega)$ on a measurable set $\Omega \subset \RRn$,   the variable exponent spaces $L^{p(\cdot)}(\Omega)$ and the Orlicz spaces $L^A(\Omega)$ have a prominent position. Loosely speaking, the former are obtained by replacing the function $t^p$  with a function $t^{p(x)}$, where the exponent $p(x)$   may depend on the space variable $x$. The latter surface when the role of the power $t^p$ is played by a more general Young function $A(t)$, namely a nonnegative convex function vanishing at $0$. 

These two families of spaces are distinguished members of the class of Musielak-Orlicz spaces $L^{\phi (\cdot)}(\Omega)$, also known as generalized Orlicz spaces. They are built upon generalized Young functions $$\phi : \RR \times [0, \infty) \to [0, \infty].$$ These functions are measurable in the variable $x\in \RRn$ for each fixed $t\in [0, \infty)$, and are Young functions in the  variable  $t\in [0, \infty)$ for a.e. fixed  $x\in \RRn$.

Accordingly, when $\Omega$ is open, the  homogeneous Musielak-Orlicz-Sobolev spaces $V^{1,\phi (\cdot)}(\Omega)$
 generalize the standard Sobolev spaces and consist of those functions whose weak derivatives belong to   $L^{\phi (\cdot)}(\Omega)$. 

The theory of Musielak-Orlicz spaces and of their Sobolev counterparts has been considerably developed in recent years. 
They were systematically introduced in \cite{Mus83}, although they already appear in the earlier contribution \cite{Nak50}. Recent comprehensive monographs on this topic are \cite{HarHas19book, MenLan19}.
Special classes of Musielak-Orlicz spaces are the subject of the books \cite{Cruz13, DieHarHasRuz17}. The interest in these spaces is especially
motivated by applications to the study of partial differential equations driven by nonstandard nonlinearities. A sample of results on the regularity of solutions to elliptic equations and variational problems in Musielak-Orlicz-Sobolev spaces can be collected from the papers \cite{Baroni18, berselli16, schw15, Byun22, Chl21, colombo15,  De22, Ho23, Ok22}  and 
the book \cite{ChlGwiSwieWro21}.

Despite extensive investigations on Musielak-Orlicz spaces, the picture of Sobolev-type embeddings  -- a fundamental tool in the applications mentioned above --   is still incomplete.  
A sharp embedding theorem is available for classical Orlicz-Sobolev spaces, associated with functions $\phi$ independent of $x$, and hence of the form
\begin{align}\label{A}
\phi(x,t) =A(t)
\end{align}
for some Young function $A$. This embedding is established in \cite{cianchi1, cianchi2} and is of use in the analysis of partial differential equations whose nonlinearities are not polynomial.

 An optimal embedding for functions $\phi$ of plain variable exponent type, namely  functions obeying
\begin{align}\label{variable}
\phi(x,t) =t^{p(x)},
\end{align}
is also well-known.   
Such an embedding was first proved in \cite{Diening04} and extended in \cite{Ed00, Ed02}. Its most general form was established in \cite{HarHas08Sobolev1n}, and can also be found in  \cite[Section 8.3]{DieHarHasRuz17}. This embedding is known to hold if the function $p(\cdot)$ satisfies a local log-H\"older continuity property, as well as a logarithmic decay assumption near infinity. On the other hand, 
it may fail if these assumptions are dropped, as demonstrated by examples from \cite[Proposition 8.3.7]{DieHarHasRuz17}.
Variable exponent Sobolev spaces enter various mathematical models of physical phenomena, such as the theory of
electrorheological fluids.

Sobolev embeddings for Musielak-Orlicz spaces associated with generalized Young functions  defined as
\begin{align}\label{double}
 \phi(x,t) = t^p + a(x)t^q
\end{align} 
are obtained in \cite{Ho23} under suitable hypotheses on the function $a(\cdot)$.  
Functions of this kind have a role in the formulation of so-called double-phase variational problems, which have been the object of renewed attention in recent years.

Embeddings for Musielak-Orlicz-Sobolev spaces built upon generalized Young functions $\phi$ of other specific forms have been considered in various contributions. Results for functions $\phi$ of general type can also be found in the literature -- see e.g. \cite{CruHas18extrapolation, Fan12embedding-musielak, HarHas08Sobolev1n, HarHas17riesz, HarHas19book, MiTaShi23}. However, they require unessential restrictions on $\phi$  and need not yield optimal conclusions. In particular, the  result of 
\cite{Fan12embedding-musielak} 
necessitates a Lipschitz dependence of $\phi$ in the $x$ variable, and, more importantly, in the case of functions $\phi$ independent of $x$ it just reproduces the non-sharp Sobolev embedding from \cite{Don71}.

The purpose of the present paper is to offer an embedding theorem for general Musielak-Orlicz-Sobolev spaces.
The Sobolev conjugate $\phi_n$ of 
any generalized Young function $\phi$ satisfying some natural assumptions on the dependence on the $x$-variable is exhibited. 
Such a Sobolev conjugate is sharp,  in the sense that the receipt that yields $\phi_n(x,\cdot)$ from $\phi(x,\cdot)$, for each fixed $x$, is exactly the same that produces the optimal Sobolev conjugate of a classical Young function, independent of $x$, in the Orlicz-Sobolev embedding.
The function $\phi_n$ is a generalized Young function   such that 
$$ V^{1,\phi (\cdot)}_{d}(\RRn) \embedding L^{\phi_n (\cdot)}(\RRn).$$
Namely, the inequality
$$
\|u\|_{L^{ \phi_n(\cdot)}(\RRn)}\leq c \|\nabla u\|_{L^{\phi(\cdot)}(\RRn)}
$$
holds for    some constant $c$ and for every $u \in V^{1,\phi(\cdot)}_{d}(\RRn)$,
Here, the arrow $\lq\lq \embedding "$ stands for continuous embedding, and $V^{1,\phi (\cdot)}_{d}(\RRn)$ denotes the space of those functions $u\in V^{1,\phi (\cdot)}(\RRn)$ which decay near infinity, in the sense that the Lebesgue measure of the level set $\{|u|>t\}$ is finite for every $t>0$. The assumptions to be imposed on the function $\phi$ are standard in the current theory of Musielak-Orlicz-Sobolev spaces, as 
recorded in \cite{HarHas19book}. They
 amount to the non-degeneracy and non-singularity in the $x$ variable for $t=1$,  to quantitative information on local oscillation in the $x$ variable, and to a suitable asymptotic behavior as $|x|\to \infty$.  These assumptions are a natural extension of those customarily requested in the basic case of variable exponent Sobolev spaces.

 Sobolev-Poincar\'e type inequalities, and ensuing Sobolev embeddings, on open sets $\Omega \subset \RRn$  with finite measure are also offered. Because of the finiteness of the ground domain, only the behavior of the generalized Young functions defining the Musielak-Orlicz norms is relevant in these inequalities. The Sobolev conjugate of $\phi$ can thus be constructed after modifying $\phi$ near $0$. This enables one to avoid unnecessary decay conditions on $\phi$.
The resultant Sobolev conjugate is denoted by~$\phi_{n,\diamond}$. 
The Sobolev-Poincar\'e  inequality
in
$V^{1,\phi(\cdot)}_0(\Omega)$, the subspace of those functions from $V^{1,\phi(\cdot)}(\Omega)$ which vanish, in a suitable sense, on $\partial \Omega$, reads
$$
\|u\|_{L^{  \phi_{n, \diamond}(\cdot)}(\Omega)}\leq c \|\nabla u\|_{L^{\phi(\cdot)}(\Omega)}
$$
 for  some constant $c$ and every  $u \in V^{1,\phi(\cdot)}_0(\Omega)$.
\\ A parallel mean-value Sobolev-Poincar\'e  inequality of the form
$$
\|u-u_\Omega\|_{L^{  \phi_{n, \diamond}(\cdot)}(\Omega)}\leq c \|\nabla u\|_{L^{\phi(\cdot)}(\Omega)}
$$
holds in sufficiently regular bounded domains $\Omega$, for  some constant $c$ and every  $u \in V^{1,\phi(\cdot)}(\Omega)$.  Here, $u_\Omega$ stands for the mean value of $u$ over $\Omega$. 
Bounded John domains $\Omega$ are admissible, for instance. Recall that the class of John domains includes, in particular, Lipschitz domains and domains with the interior cone property.
\\ Compact embeddings for the spaces $V^{1,\phi(\cdot)}_0(\Omega)$ and $V^{1,\phi(\cdot)}(\Omega)$ are discussed as well.

Our embeddings recover the optimal known results for specific choices of the function $\phi$. In particular, they reproduce the Sobolev embeddings for functions $\phi$ as in \eqref{A}, \eqref{variable}, and \eqref{double}, and also allow for critical and supercritical values of parameters associated with these functions.

A major dissimilarity between Musielak-Orlicz and classical Orlicz spaces is that, unlike the latter, the former are not rearrangement-invariant. This drawback makes some powerful tools in dealing with Sobolev type inequalities in Orlicz spaces, such as symmetrizations and rearrangement-based interpolation techniques, useless in the Musielak-Orlicz environment. 
As in most of the available proofs of Sobolev embeddings for spaces from the  Musielak-Orlicz class, we  resort to an approach relying on a 
 weak-type estimate for the Riesz potential operator $I_\alpha$ in the same spaces. A critical step in our proof amounts to a new estimate of this kind, which
takes the form:
$$
\int_{\{|I_\alpha f|>t\}}\phi_{\frac n\alpha}(x,ct)\, dx \leq \int_{\RRn}\phi(x, |f(x)|)\, dx
$$
for every (suitably normalized) generalized Young function $\phi$, for some positive constant $c$ and for every function $f$ such that $\|f\|_{L^{\phi(\cdot)}(\RRn)}\leq 1$. Here, $\alpha \in (0,n)$,  and $\phi_{\frac n\alpha}$ is the sharp Sobolev conjugate of $\phi$ of order $\alpha$. 
The function  $\phi_{\frac n\alpha}$ agrees with $\phi_n$ if $\alpha =1$, and reproduces, for each fixed $x$, the Sobolev conjugate for optimal weak-type inequalities for  $I_\alpha$ in Orlicz spaces exhibited in \cite{cianchi_JLMS}. Moreover,
our estimate includes, as a special case,  a well-known inequality in variable exponent Lebesgue spaces \cite{Diening04}.  It also augments weak-type inequalities for Riesz potentials in Musielak-Orlicz spaces contained in \cite{CruHas18extrapolation, HarHas08Sobolev1n, HarHas17riesz}.
 
Thanks to a representation formula for Sobolev functions in terms of Riesz potentials of order $\alpha =1$, the weak-type inequality for these potentials implies a parallel weak-type Sobolev inequality in  Musielak-Orlicz spaces. A nowadays standard discretization-truncation argument introduced by V.Maz'ya \cite{Ma61, Mabook} then upgrades the weak-type Sobolev inequality to a strong one.  The relevant representation formula, though classical for compactly supported functions,  cannot apparently be found in the literature for functions decaying near infinity in the very weak sense specified above. Its proof requires some additional argument and it is hence stated in a separate lemma, of possible independent interest.

\section{Generalized Young functions and Musielak-Orlicz spaces}\label{background}

This section is devoted to basic notions about generalized Young functions and the associated Musielak-Orlicz spaces. 
We rely upon the modern foundations of the theory of these spaces developed in the monographs \cite{DieHarHasRuz17,HarHas19book}, to which we refer the reader for more details.

A Young function is a function $A\,:\, [0,\infty) \to [0,\infty]$ which is convex, continuous, non-constant and such that $A(0)=0$.  The following representation formula holds
\begin{align*}
  A(t) = \int_0^t a(\tau)\,d\tau,
\end{align*}
where $a: [0, \infty) \to [0, \infty]$ is a non-decreasing  function. 
 \\
By~$A^{-1} : [0, \infty) \to [0, \infty)$ we denote the left-continuous inverse of $A$, given by $$A^{-1}(t) = \inf \set{\tau \geq 0\,:\, A(\tau) \geq t} \quad \text{ for $t \geq 0$.}$$

\smallskip
 \par\noindent
A function $\phi\,:\, \RRn \times [0, \infty) \to [0,\infty]$ is called
 a generalized Young function  if
\begin{enumerate}
\item [(i)] the function $\phi(x,\cdot)$ is a Young function for a.e.  $x \in \RRn$,
\item [(ii)]  the function $\phi(\cdot, t)$ is measurable for every~$t \geq 0$.
\end{enumerate}
Notice that
\begin{equation}\label{oct25}
s\,\phi(x, t) \leq   \phi(x,st) \quad \text{if $s\geq 1$ and $t \geq 0$,}
\end{equation}
for $x\in \RRn$.
\\ For each $x \in \RRn$, we denote by $\phi^{-1}(x, \cdot)$ the  left-continuous inverse of the Young function $\phi(x, \cdot)$. 
One has that
$$\phi(x,\phi^{-1}(x, t))=t\quad \text{and} \quad \phi^{-1}(x, \phi(x,t))\leq t \quad  \text{for $x\in \RRn$ and $t \geq 0$.}$$
The Young conjugate of  a generalized Young function $\phi$ is  denoted by 
$\widetilde \phi$ and defined as 
$$\widetilde \phi (x,t) = \sup\{\tau t - \phi(x,\tau): \tau\geq 0\} \quad \text{for $x\in\RRn$ and  $t\geq 0$.}$$
It is also a generalized Young function and 
\begin{equation}\label{6}
t \leq \phi^{-1}(x, t) \tp^{-1}(x,t) \leq 2t \quad \text{for $x\in \RRn$ and $t\geq 0$.}
\end{equation}
The function $\phi$ is said to satisfy the $\Delta_2$-condition if there exists a constant $c$ such that
\begin{equation}\label{delta2}
\phi(x,2t) \leq c \,\phi (x,t)\quad \text{for $x\in\RRn$ and  $t\geq 0$.}
\end{equation}
In what follows, we write
\begin{equation}\label{oct1} \phi(x,t) \approx \psi (x,t)
\end{equation}
to denote that there exist positive constants $c_1$ and $c_2$ such that $c_1  \phi(x,t) \leq \psi (x,t) \leq c_2  \phi(x,t)$ for $x\in \RRn$ and $t\geq 0$. This relation differs from the equivalence in the sense of Young functions, denoted by
\begin{equation}\label{oct2} \phi(x,t) \simeq \psi (x,t),
\end{equation}
which means that $ \phi(x,c_1 t) \leq \psi (x,t) \leq  \phi(x,c_2t)$. Relations \eqref{oct1} and \eqref{oct2} are said to hold near zero or near infinity if the defining inequalities just hold, uniformly in $x$, for small or large values of $t$, respectively. 
These relations will also be employed between functions that are not generalized Young functions.  Observe that, thanks to property \eqref{oct25}, if $\phi$ and $\psi$ are generalized Young functions such that $\phi \approx \psi$, then  $\phi \simeq \psi$ as well.
{\color{black} \\ We say that the function $\vartheta$   grows essentially more slowly than $\phi$ near infinity in a set $\Omega \subset \RRn$ if
 \begin{align}\label{moreslowly}
      \lim_{t \to \infty} \esssup_{x \in \Omega}\frac{\vartheta(x,c t)}{\phi(x,t)} &= 0 \qquad \text{for every $c>0$.}
  \end{align}}
\smallskip
\par
Let $\Omega$ be a measurable set in $\RRn$. Denote by $L^0(\Omega)$ the set of measurable functions in $\Omega$. 
The   Musielak-Orlicz space~$L^{\phix}(\Omega)$, associated with a generalized Young function  ~$\phi$, is defined as
\begin{align*}
  L^\phix(\Omega)  =\bigg\{ u \in L^0(\Omega)\,:\,  \int_{\Omega} \phi\bigg(x,\frac{\abs{u(x)}}\lambda \bigg)\,dx<\infty \text{ for some } \lambda >0\bigg\}.
\end{align*}
The space $L^{\phix}(\Omega)$ is a Banach space when equipped with the norm
\begin{align}\label{norm}
  \norm{u}_{ L^\phix(\Omega) } = \inf \bigg\{ \lambda > 0\,:\, \int_{\Omega} \phi\bigg(x,\frac{\abs{u(x)}}\lambda \bigg)\,dx \leq 1\bigg\}.
\end{align}
The following H\"older type inequality holds:
\begin{equation}\label{holder}
\int_\Omega |u(x)v(x)|\, dx \leq 2 \|u\|_{L^{\phix}(\Omega)}  \|v\|_{L^{\tp (\cdot)}(\Omega) }
\end{equation}
for $u \in L^{\phix}(\Omega)$ and $v \in L^{\tp (\cdot)}(\Omega) $. Furthermore,
\begin{equation}\label{revholder}
\sup_{u \in L^{\phix}(\Omega)}\frac{\int_\Omega |u(x)v(x)|\, dx}{ \|u\|_{L^{\phix}(\Omega)} }\geq  \|v\|_{L^{\tp (\cdot)}(\Omega) }
\end{equation}
for $v \in L^{\tp (\cdot)}(\Omega)$.

Musielak-Orlicz spaces are, in particular, normed function lattices on $\Omega$. Recall that a normed function lattice $X(\Omega)$  is a normed linear space of functions in $L^0(\Omega)$, endowed with a norm $\|\cdot\|_{X(\Omega)}$ such that
\begin{equation}\label{lattice}
\text{if $|u|\le |v|$ a.e. in $\Omega$, then $\|u\|_{X(\Omega)}\le\|v\|_{X(\Omega)}$.}
\end{equation}

Additional assumptions are needed on generalized Young functions for several classical results from Real and Harmonic Analysis to hold in the associated Musielak-Orlicz spaces. The requirements to be imposed in the model case of variable exponent spaces are thoroughly exposed in the monograph \cite{DieHarHasRuz17}. Their extension to the general  Musielak-Orlicz space depends upon the following definitions, which were introduced in \cite{HarHas19book}.

\begin{definition}
  \label{def:axioms}
  A generalized Young function $\phi$ is said to satisfy condition:
  \begin{enumerate}
  \item[{\rm (A0)}] if there exists~$\beta \in (0,1]$ such that
    \begin{align*}
      \beta \leq \phi^{-1}(x,1) \leq \frac 1 \beta
    \end{align*}
    for a.e.~$x \in \RRn$;
  \item[{\rm (A1)}] if there exists~$\beta \in (0,1]$   such that
    \begin{align*}
      \beta \phi^{-1}(x,t) \leq \phi^{-1}(y,t)
    \end{align*}
    for every $t \in [1, \frac{1}{\abs{B}}]$, for  a.e.~$x,y \in B$, and   every ball~$B \subset \RRn$ with $\abs{B} \leq 1$;
  \item[{\rm (A2)}] if for every~$s>0$ there exist~$\beta \in (0,1]$  and $h \in L^1(\RRn) \cap L^\infty(\RRn)$ such that
    \begin{align*}
      \beta \phi^{-1}(x,t) \leq \phi^{-1}(y,t)
    \end{align*}
    for a.e.~$x,y \in \RRn$ and every $t \in [h(x)+h(y),s]$.
  \end{enumerate}
\end{definition}
Condition~(A0) is usually called the weight condition, (A1) the local continuity condition, and (A2) the decay condition. Note that, in spite of its name, condition (A1) just provides information on the local oscillation of $\phi$ in the $x$ variable, but does not entail its continuity. This terminology is inspired by the special setting of variable exponent functions.

{\color{black} Condition (A0) also ensures that, if $E$ is any set in $\RRn$ with finite Lebesgue measure $|E|$, then $\|\chi_E\|_{L^\phix (\RRn)}<\infty$. Here, $\chi_E$ denotes the characteristic function of $E$.
Thanks to inequalities \eqref{6}, condition (A0),  with $\frac 1\beta$ replaced with $\frac 2\beta$, is inherited by the function $\widetilde \phi$. Thus, as a consequence of inequality \eqref{holder}, 
\begin{equation}\label{nov150}
\int_{E}|u|\,dx <\infty \quad \text{ if}\quad  \|u\chi_E\|_{L^\phix (\RRn)}<\infty.
\end{equation}
In particular, if $\phi$ satisfies (A0) and $|\Omega|<\infty$, then
\begin{equation}\label{nov151}
L^\phix(\Omega) \embedding L^1(\Omega).
\end{equation}

The function $\phi_\infty : [0, \infty) \to [0, \infty]$, associated with a generalized Young function $\phi$ by
\begin{equation}\label{phiinf}
\phi_\infty (t) = \limsup_{|x|\to \infty} \phi(x,t) \qquad \text{for $t \geq 0$,}
\end{equation}
plays a role in equivalent formulations of condition (A2). If $\phi_\infty$ is non-degenerate, in the sense that it is neither constantly equal to $0$ nor to $\infty$, then it is a Young function. This is certainly the case when $\phi$ satisfies condition~(A0).
 \\ As shown in \cite[Lemma~4.2.9]{HarHas19book} and \cite{HH2023arXiv}, 
for a generalized Young function $\phi$ satisfying condition (A0),  the following conditions  are equivalent to (A2).

\begin{definition}
  \label{def:axioms'} 
 A generalized Young function  $\phi$ is said to satisfy condition:
\begin{enumerate}
\item[{\rm (A2')}] if 
  \begin{align*}
    L^\phix(\RRn) \cap L^\infty(\RRn) = L^{\phi_\infty}(\RRn) \cap L^\infty(\RRn),
  \end{align*}
  with equivalence of norms, where   $\phi_\infty$ is the Young function defined by \eqref{phiinf};
\item[{\rm (A2'')}] if there exist $\beta \in (0,1]$ and  $h \in L^1(\RRn) \cap L^\infty (\RRn)$  such that for a.e. $x\in \RRn$
  \begin{align*}
 \phi(x, \beta t) \leq \phi_\infty (t) + h(x) \,\, \text{if $\phi_\infty (t) \in [0,1]$} \quad 
\text{and} \quad 
 \phi_\infty(\beta t) \leq \phi(x,t) + h(x) \,\,\text{if $\phi (x,t) \in [0,1]$.}
  \end{align*}
\item[{\rm (A2''')}]  if for every $\sigma >0$ there exist $\beta \in (0,1]$ and $h \in L^1(\RRn) \cap L^\infty(\RRn)$, with $h \geq 0$, such for a.e. $x, y \in \RRn$
\begin{align*}
\beta \phi^{-1}(x,t) \leq \phi^{-1}(y,t + h(x) +h(y)) \quad  \text{if $t \in [0, \sigma]$.}
\end{align*}
\end{enumerate}
\end{definition}}

Let us briefly discuss the customary generalized Young functions defined as in \eqref{variable} and \eqref{double}
in connection with properties (A0), (A1) and (A2) introduced above. We refer the reader to \cite{HarHas19book} for additional examples. 

\begin{example}{\rm\bf{ [Variable exponents]}} \label{exvariable}{\rm
The best known instances of generalized Young functions, depending on the $x$-variable in a non trivial way, are provided by variable powers, namely 
 functions given by
\begin{equation}\label{var0}
\phi(x,t) = t^{p(x)} \quad \text{for  $x\in \RRn$ and $t \geq0$,}
\end{equation}
where
$$p\,:\, \RRn \to [1,\infty).$$  
They are at the basis of the definition of the variable exponent Lebesgue spaces, whose systematic analysis was initiated in the paper  \cite{Ko91}.
\\ 
 If the function $p$ satisfies a $\log$-H\"older continuity continuity property of the form
\begin{equation}\label{var1}
 \biggabs{\frac{1}{p(x)} - \frac{1}{p(y)}} \leq \frac{c}{\log(e+ \frac{1}{\abs{x-y}})} \quad \text{for  $x,y \in \RRn$,}
\end{equation}
for some constant $c$,
then $\phi$
is a generalized Young function fulfilling properties (A0) and  (A1) -- see  \cite[Section 7.1]{HarHas19book}. 
\\ If, in addition, 
 the logarithmic decay condition near infinity
\begin{equation}\label{var2}
  \biggabs{\frac{1}{p(x)} - \frac{1}{p_\infty}} \leq \frac{c}{\log(e+ \abs{x})} \quad \text{for  $x\in \RRn$,}
\end{equation}
holds for some constants $p_\infty \in [1, \infty)$ and  $c>0$, then the function $\phi$ also satisfies property (A2).  In particular,   under assumptions \eqref{var1} and \eqref{var2}, we have that $p\in L^\infty(\RRn)$.
\\ Observe that condition \eqref{var2} implies Nekvinda's condition introduced in \cite{Nek04}. 
\\
The same assertions are true for the variant given by
$$\phi(x,t) =  \tfrac{1}{p(x)} t^{p(x)}\quad \text{for  $x\in \RRn$ and $t \geq0$,}$$
which is of common use in the formulation of variational problems associated with variable exponent functionals. 
}
\end{example}

\begin{example}{\rm\bf{ [Double-phase]}} \label{doubleex}{\rm Another popular class of generalized Young functions consists of so-called double-phase Young functions. They have the form
\begin{equation}\label{double1}
\phi(x,t)= t^p + a(x)t^q \quad \text{for  $x\in \RRn$ and $t \geq0$,}
\end{equation}
for some exponents $q >p \geq 1$ and for a nonnegative function $a\in L^\infty(\RRn)$. As shown in \cite[Section~7.2]{HarHas19book},
 a function $\phi$ of this kind satisfies properties (A0) and (A2), with $h=0$. It  also fulfills property (A1) if and only if
\begin{equation}\label{double5}
a(x) \leq c\Big(a(y) + |x-y|^{\frac{n}{p}(q-p)}\Big)\quad \text{for  $x,y\in \RRn$,}
\end{equation}
for some constant $c>0$.  In particular, condition \eqref{double5} holds, provided that
\begin{equation}\label{double2}
q \leq \frac{np}{n-1},
\end{equation}
 and 
 \begin{equation}\label{double3}
a \in C^{0,\frac{n}{p}(q-p)}(\RRn).
\end{equation}
 An alternate version of the function \eqref{double1}, which still enjoys properties (A0), (A1), and (A2) under the same assumptions, is obtained by replacing the sum with the maximum on the right-hand side of \eqref{double1}. This results in the function 
\begin{equation}\label{double4}
 \phi(x,t)= \max\{t^p, a(x)t^q\} \quad \quad \text{for  $x\in \RRn$ and $t \geq0$.}
\end{equation}
  A combination of the functions from \eqref{var0} and \eqref{double1} yields variable exponent double-phase functions given by
\begin{equation}\label{doublevar1}
\phi(x,t)= t^{p(x)} + a(x)t^{q(x)} \quad \quad \text{for  $x\in \RRn$ and $t \geq0$.}
\end{equation}
They have been examined, also in connection with conditions (A0), (A1), and (A2), in the papers  \cite{crespo22} and \cite{Ho23}.}
\end{example}

\smallskip
A handy subclass of generalized Young functions consists of normalized functions according to the next definition patterned on   \cite{HarHas17riesz}.
\begin{definition}\label{normalized} A  generalized Young function $\phi$  will be called normalized   if:
\\ (i)
    \begin{align}\label{july3}
       \phi^{-1}(x,1)=1 \quad \text{for a.e. $x\in \RRn$;}
    \end{align}
\\ (ii)
 there exists~$\beta \in (0,1]$   such that
    \begin{align}\label{july4}
      \beta \phi^{-1}(x,t) \leq \phi^{-1}(y,t)
    \end{align}
    for every $t \in [0, \frac{1}{\abs{B}}]$,  for  a.e.~$x,y \in B$, and   every ball~$B \subset \RRn$;
 \\ (iii)
    \begin{align}\label{july5}
        \phi(x,t) =       \phi _\infty(t) \quad \text{for $t\in [0,1]$.}
    \end{align}
\end{definition}

The following properties of normalized generalized Young functions follow from their definition, via inequalities \eqref{6}.

\begin{proposition}\label{feb1}
Assume  that   $\varphi$ is a normalized generalized Young function, and  let $\beta\in (0,1]$ be the constant appearing in Definition  \ref{normalized}.  Then:
\\ (i) \begin{equation}\label{feb2}
  \phi (x,t) \leq   \phi (y,\tfrac t\beta)
\end{equation}
for  every $t \in [0,\beta \phi^{-1}(y, \abs{B}^{-1})]$,  for  a.e.~$x,y \in B$, and   every ball~$B \subset \RRn$;
 \\ (ii)
\begin{equation}\label{feb3}
\tfrac \beta 2 \widetilde \phi ^{-1}(x,t) \leq \widetilde \phi ^{-1}(y,t)
\end{equation}
for every $t \in [0, \abs{B}^{-1}]$, for  a.e.~$x,y \in B$, and   every ball~$B \subset \RRn$;
\\ (iii)
\begin{equation}\label{25}
  \widetilde \phi (x,t) \leq   \widetilde\phi (y,\tfrac 2\beta t)
\end{equation}
for every $t \in [0,\frac \beta 2 \widetilde \phi^{-1}(y, \abs{B}^{-1})]$, for  a.e.~$x,y \in B$, and   every ball~$B \subset \RRn$.
\end{proposition} 

A useful result from \cite{HarHas17riesz} ensures that any  generalized Young function $\phi$ satisfying conditions 
(A0), (A1), and (A2) can be replaced with a normalized generalized Young function $\overline \phi$ without altering the associated Musielak-Orlicz space, up to equivalent norms. The recipe is the following. First, 
define the generalized Young function $\phi_0 : \RRn \times [0, \infty) \to [0, \infty]$ as 
\begin{align}\label{july8}
\phi_0(x,t)=
 \max\big\{\phi\big(x,\phi^{-1}(x,1)t\big), 2t-1\big\} \quad \text{for $x\in \RRn$ and $t \geq 0$.}
\end{align}
Then, the function  $\overline \phi : \RRn \times [0, \infty) \to [0, \infty]$ is given by
\begin{align}\label{july7}
      \overline \phi(x,t)= \begin{cases} 2\phi_0(x,t)-1 \quad &\text{if $t \geq 1$}
\\ \limsup_{|x|\to \infty} \phi_0(x,t) \quad &\text{if $0\leq t <1$,}
\end{cases}
\end{align}
for $x\in \RRn$. 

\begin{proposition}
{\rm\bf{ \cite[Proposition 4.2]{HarHas17riesz}}}
 \label{replacephi} Let $\phi$ be a generalized Young function satisfying conditions 
(A0), (A1) and (A2).  Then, the function $\overline \phi$ defined by \eqref{july7} is a normalized generalized Young function, which satisfies condition \eqref{july4} with the same $\beta$ as in condition (A1) for $\phi$. Moreover, 
   \begin{align} \label{july6}
L^{\phix}(\RRn) = L^{\overline {\phi}(\cdot)}(\RRn),
    \end{align}
up to equivalent norms,   with equivalence constants depending on the constant $\beta$ from conditions (A0), (A1), and on the function $h$ from condition (A2).
\end{proposition}

Notice that $\overline \phi$ is a generalized Young function fulfilling condition (A0) even if $\phi$ does not satisfy conditions 
(A0), (A1), and (A2). These are needed for $\overline \phi$ to be normalized and for equation \eqref{july6} to hold.

As pointed out in \cite{HarHas17riesz},  despite property \eqref{july6} the functions $\phi$ and $\overline \phi$ need not be globally equivalent, in general. 
However, as shown in the same paper, one has that
 \begin{equation}\label{july12}
 \phi(x,\beta t) \leq   \overline \phi(x,t) \leq \phi\big(x, \tfrac 4\beta t\big)\qquad \text{if  $t\geq 1$,}
\end{equation}
and
\begin{equation}\label{july9}
  \phi_\infty \big(\beta t\big)\leq  \overline \phi  (x,t) \leq \phi_\infty (\tfrac {2t}\beta) \qquad \text{if   $0\leq t <1$,}
\end{equation}
for $x\in \RRn$,
where $\beta$ is the constant appearing in property (A0)  for the function $\phi$. 
\\ Consequently,  if  $\phi$ fulfills both condition (A0) and the $\Delta_2$-condition, then the 
 function $\phi^\circ : \RRn \times [0, \infty) \to [0, \infty]$ obeying
\begin{align}\label{july13}
       \phi^\circ (x,t)= \begin{cases} \phi (x,t) \quad &\text{if $t\geq 1$}
\\ \phi_\infty (t)\quad &\text{if  $0\leq t <1$,}
\end{cases}
\end{align}
for $x\in \RRn$, is such that
\begin{equation}\label{oct47}
\phi^\circ \approx \overline \phi.
\end{equation}

When dealing with Musielak-Orlicz spaces defined on sets of finite measure, assumption (A2) on the generating  function $\phi$ is immaterial. The last result of this section shows that, if $\phi$ just fulfills assumptions (A0) and (A1), then there exists a normalized generalized Young function $\widehat \phi$ such that $L^{\widehat {\phi}(\cdot)}(\Omega)$ and  $L^{\phix}(\Omega)$ agree, up to equivalent norms, for every measurable set  $\Omega \subset \RRn$ with $|\Omega|<\infty$.
The function $\widehat \phi$ is defined as 
\begin{equation}\label{july11}
      \widehat \phi(x,t)= \begin{cases} 2\phi_0(x,t)-1 \quad & \text{if $t \geq 1$}
\\ t \quad &\text{if $0\leq t<1$,}
\end{cases}
\end{equation}
for $x \in \RRn$,
where $\phi_0$ is given by equation \eqref{july8}.

\begin{proposition}\label{replacephiomega} Let $\phi$ be a generalized Young function satisfying properties 
(A0) and (A1).
Then the function $\widehat \phi$ given by \eqref{july11} is a normalized generalized Young function{\color{black}, which satisfies condition \eqref{july4} with the same $\beta$ as  in condition (A1) for $\phi$}. Moreover,
\begin{align}\label{july19}
L^{\widehat \phi (\cdot)}(\Omega) = L^{\phi (\cdot)}(\Omega),
\end{align} 
up to equivalent norms,  for every measurable set  $\Omega \subset \RRn$ with $|\Omega|<\infty$.
\end{proposition}
\begin{proof} Observe that
\begin{equation}\label{oct45}
\widehat \phi (x,t)= \overline \phi (x,t) \qquad \text{if $t \geq 1$}
\end{equation}
and $x \in \RRn$. Since $\overline \phi$ is a normalized generalized Young function satisfying condition \eqref{july4} with the same $\beta$ as  in condition (A1) for $\phi$,
the fact that $\widehat \phi$ enjoys the same property  can be verified as a consequence of identities \eqref{oct45} and 
$$\widehat \phi (x,t)= t \qquad \text{if $0\leq t <1$}$$
and   $x \in \RRn$. 
\\ As far as 
property \eqref{july19} is concerned, identity \eqref{oct45}
and  inequalities \eqref{july12} yield:
  \begin{equation}\label{oct37}
  \widehat \phi(x,t) \leq \phi\big(x, \tfrac 4\beta t\big)\qquad \text{if  $t\geq 1$,}
\end{equation}
and
  \begin{equation}\label{oct39}
 \phi(x, t) \leq  \widehat \phi(x,\tfrac t\beta) \qquad \text{if  $t\geq \beta$,}
\end{equation}
 for $x \in \RRn$. On the other hand,
  \begin{equation}\label{oct38}
 \widehat \phi (x,t) \leq  \widehat \phi (x,1) = 1\qquad \text{if  $0\leq t< 1$.}
\end{equation}
and
  \begin{equation}\label{oct40}
 \phi(x, t) \leq  \phi(x,\beta)\leq 1 \qquad \text{if  $0 \leq t <\beta$,}
\end{equation}
where the last inequality holds thanks to property (A0) of $\phi$. Combining inequalities \eqref{oct37} and \eqref{oct38} tells us that
  \begin{equation}\label{oct41}
  \widehat \phi(x,t) \leq \phi\big(x, \tfrac 4\beta t\big) + 1\qquad \text{for $x \in \RRn$ and  $t\geq 0$,}
\end{equation}
whereas inequalities \eqref{oct39} and \eqref{oct40} yield
  \begin{equation}\label{oct42}
 \phi(x,t) \leq  \widehat \phi(x,\tfrac t\beta) + 1 \qquad \text{for $x \in \RRn$ and  $t\geq 0$,}
\end{equation}
Equation \eqref{july19} follows from inequalities \eqref{oct41} and \eqref{oct42}, via  \cite[Theorem 2.8.1]{DieHarHasRuz17}.
\end{proof}

Under the assumption that $\phi$ fulfills both condition (A0) and the $\Delta_2$-condition, consider the function $\phi^\bullet : \RRn \times [0, \infty) \to [0, \infty]$ defined as
\begin{align}\label{oct48}
       \phi^\bullet (x,t)= \begin{cases} \phi (x,t) \quad &\text{if $t\geq 1$}
\\t \quad &\text{if  $0\leq t <1$,}
\end{cases}
\end{align}
for $x\in \RRn$. Although, like $\phi^\circ$, the function $\phi^\bullet$ may not even be continuous at $t=1$, one has that
\begin{equation}\label{oct49}
\phi^\bullet \approx \widehat \phi.
\end{equation}
This is a consequence of equations \eqref{july12} and \eqref{july9}, and of the fact that $\widehat \phi(x,t)= \overline \phi(x,t)$ if $t\geq 1$ and $x \in \RRn$.

\section{Main results}\label{sec:main}

 Assume that $\Omega$ is an open set in $\RRn$ and $\phi$ is a generalized Young function. The homogeneous Musielak-Orlicz-Sobolev space $V^{1,\phix}(\Omega)$ is defined as 
\begin{equation}\label{june1}
V^{1,\phix}(\Omega) = \{u\in W^{1,1}_{\rm loc}(\Omega): \, |\nabla u|\in  L^{\phix}(\Omega)\}.
\end{equation}
{\color{black} If $\Omega$ is connected and $G$ is a bounded open set such that  $\overline G \subset \Omega$, then the functional 
\begin{equation} \label{nov153}
\|u\|_{L^1(G)}+ \|\nabla u\|_{L^\phix(\Omega)}
\end{equation} 
defines a norm on $V^{1,\phix}(\Omega)$. One can show that different choices of the set $G$ result in equivalent norms.}
\\
 The subspace of functions which vanish on $\partial \Omega$   is suitably given by 
\begin{equation}\label{june1vanish}
V^{1,\phix}_0(\Omega) = \{u\in W^{1,1}_{\rm loc}(\Omega): \, \text{the extension to $\RRn$ of $u$ by $0$ outside $\Omega$ belongs to $V^{1,\phix}(\RRn)$} \},
\end{equation}
{\color{black}and can equipped with the norm
\begin{equation}\label{nov154}
\|\nabla u\|_{L^\phix(\Omega)}.
\end{equation}}
When $\Omega =\RRn$, we also define the space
\begin{equation}\label{june2}
V^{1,\phix}_d(\RRn) = \{u \in V^{1, \phix}(\RRn): \, |\{|u|>t\}|<\infty\,\, \text{for every $t>0$}\}.
\end{equation}
Heuristically speaking, 
$V^{1,\phix}_d(\RRn)$ is the subspace of $V^{1,\phix}(\RRn)$  of those functions decaying near infinity in a weakest possible sense for a homogeneous Sobolev inequality in $\RRn$ to hold. {\color{black} The functional \eqref{nov154}, with $\Omega=\RRn$, defines a norm in $V^{1,\phix}_d(\RRn)$.
\\ When $\phi (x,t)=t^p$ for some constant $p\in [1,\infty)$, we shall simply write $V^{1,p}$ instead of $ V^{1,\phix}$ in the above notations.}

\subsection{\texorpdfstring{The inequality in $\RRn$}{The inequality in Rn}}\label{Rn}
 Our 
 embedding theorem in $\RRn$ holds for any Musielak-Orlicz-Sobolev space associated with a generalized Young function $\phi$ satisfying conditions (A0), (A1), (A2),
and
\begin{equation}\label{conv0}
\int_0 \bigg(\frac t{\phi_\infty(t)}\bigg)^{\frac 1{n-1}}\, dt <\infty.
\end{equation}
As shown in Proposition \eqref{nec} below, assumption \eqref{conv0} is indispensable when considering Sobolev inequalities in $V^{1,\phix}_d(\RRn)$.
 \\ The Sobolev conjugate  of $\phi$ is the generalized 
 Young function  $\phi_n$ defined as 
\begin{equation}\label{sobconj}
\phi_n (x,t) = \overline  \phi (x,H_n^{-1}(x,t)) \qquad \text{for $x\in\RRn$ and $t \geq 0$,}
\end{equation}
where  $H_n: \RRn \times [0, \infty) \to [0, \infty)$ is the function given by
\begin{equation}\label{H}
H_n(x,t) = \bigg(\int_0^t \bigg(\frac \tau{\overline \phi(x,\tau)}\bigg)^{\frac 1{n-1}}\, d\tau\bigg)^{\frac {1}{n'}} \quad \text{for $x \in \RRn$ and $t \geq 0$.}
\end{equation}
Here, $\overline \phi$ denotes the normalized generalized Young function associated with $\phi$ as in \eqref{july7}
and $n'= \frac n{n-1}$, the H\"older conjugate of $n$. 
\\ Thanks to equation \eqref{july9},   assumption \eqref{conv0} ensures that the function $H_n$ is finite-valued for $x \in \RRn$ and $t \geq 0$. Here,  $H_n^{-1}(x,\cdot)$ denotes the generalized left-continuous inverse of  $H_n(x,\cdot)$ for $x\in \RRn$. Note that, if 
\begin{equation}\label{convinf}
\int^\infty \bigg(\frac \tau{ \phi(x,\tau)}\bigg)^{\frac 1{n-1}}\, d\tau<\infty
\end{equation}
for some $x\in \RRn$,
then, by equation \eqref{july12}, $\lim_{t \to \infty}H_n(x,t)<\infty$. Hence, 
$H_n^{-1}(x,t)=\infty$ for $t$ exceeding this limit. The right-hand side of equation \eqref{sobconj} has  then to be interpreted as $\infty$ for these values of $t$.
 
\begin{theorem}{\rm\bf{ [Sobolev inequality in $\RRn$]}}\label{thm:main} 
Assume that $\phi $ is a  generalized Young function satisfying 
conditions
(A0), (A1), (A2) and 
\eqref{conv0}. 
Then, $$V^{1,\phi(\cdot)}_{d}(\RRn)\embedding L^{\phi_n(\cdot)}(\RRn),$$ and there exists a constant  $c=c(n, \beta, h)$  such that
\begin{equation}\label{main1}
\|u\|_{L^{\phi_n(\cdot)}(\RRn)}\leq c \|\nabla u\|_{L^{\phi(\cdot)}(\RRn)}
\end{equation}
for every $u \in V^{1,\phi(\cdot)}_{d}(\RRn)$. Here,  $\beta$ denotes the constant from  conditions (A0) and (A1), and $h$ the function from condition (A2). In particular, if $\phi= \overline \phi$, then the constant $c$ in inequality \eqref{main1} only depends on $n$ and~$\beta$.
\iffalse
\\   Moreover,  $L^{\phi_n(\cdot)}(\RRn)$ is the optimal Musielak-Orlicz  target space in \eqref{main1}, in the sense that, if inequality  \eqref{main1} holds with $L^{\phi_n(\cdot)}(\RRn)$ replaced by another Musielak-Orlicz space $L^{\psi (\cdot)}(\RRn)$, then $L^{\phi_n(\cdot)}(\RRn) \embedding L^{\psi (\cdot)}(\RRn)$.
\fi
\end{theorem}

\begin{remark}\label{rem-oct}{\rm
One can verify that replacing $\overline \phi$ with any equivalent (in the sense of the relation $\simeq$)  generalized Young function in formulas \eqref{sobconj} and \eqref{H} results in a  generalized Young function equivalent to $\phi_n$. 
For customary choices of $\phi$, such an equivalent function is immediately obtained from the expression of $\phi$. These replacements just affect the constant $c$ appearing in inequality \eqref{main1}. In fact, the same conclusion holds even if $\phi$ is replaced with an equivalent function,  which is not necessarily a generalized Young function. The norm $\|\cdot \|_{L^{\phi_n(\cdot)}(\RRn)}$ is in fact equivalent to the functional obtained after this replacement.
Thus, for instance, if $\phi$ satisfies the $\Delta_2$-condition, then, thanks to equation \eqref{oct47}, the norm $\|\cdot \|_{L^{\phi_n(\cdot)}(\RRn)}$ in inequality \eqref{main1} can be replaced with the functional generated by the use of the function $\phi^\circ$ given by \eqref{july13}, instead of $\overline \phi$, in \eqref{sobconj} and \eqref{H}.
}
\end{remark}

\begin{remark}\label{integralform} {\rm An inspection of the proof of Theorem \ref{thm:main}  will reveal that, under the same assumptions, the following  Sobolev inequality in modular form holds:
 \begin{equation}\label{integralform1}
\int_{\RRn} \phi_n(x, c |u|)\, dx  \leq   \int_{\RRn} \overline\phi(x, |\nabla u|)\, dx 
\end{equation}
for some constant  $c=c(n, \beta)$  and for every $u$ such that $ \int_{\RRn} \overline \phi(x, |\nabla u|)\, dx \leq 1$.   Inequality \eqref{integralform1} can be of use in applications to the analysis of solutions to partial differential equations. Parallel conclusions about possible replacements of the function $\overline\phi$ in the definition of the function $\phi_n$ as in Remark \ref{rem-oct} hold with regard to inequality  \eqref{integralform1}.}
\end{remark}

The following result demonstrates the necessity of condition \eqref{conv0}  for the embedding of Theorem \ref{thm:main}, and even for more general embeddings of the space $V^{1,\phi(\cdot)}_{d}(\RRn)$ into normed function lattices.

\begin{proposition}{\rm\bf{ [Necessity of condition \eqref{conv0}]}}\label{nec}
Let $\phi $ be as in Theorem \ref{thm:main}. % satisfying properties (A0), (A1), and (A2). 
Assume that there exists a normed function lattice $Y(\RRn)$ such that
\begin{equation}\label{nec1}
\|u\|_{Y(\RRn)}\leq c \|\nabla u\|_{L^{\phi(\cdot)}(\RRn)}
\end{equation}
for every $u \in V^{1,\phi(\cdot)}_{d}(\RRn)$. Then, condition \eqref{conv0} must be fulfilled.

\end{proposition}

\subsection{Inequalities in  domains}\label{domains}

 Here, we offer 
  Sobolev-Poincar\'e type inequalities in  open sets $\Omega \subset \RRn$ with finite Lebesgue measure. Of course, they imply corresponding Sobolev embeddings. Thanks to the finiteness of the measure of $\Omega$, condition (A2) on the function $\phi$ defining the Musielak-Orlicz-Sobolev space is no longer required. For the same reason, assumption \eqref{conv0} can also be dropped.  Namely, the Sobolev-Poincar\'e inequality holds whenever $\phi$ satisfies conditions (A0) and (A1).
Indeed, owing to Proposition  \ref{replacephiomega}, if $\phi$ fulfills these conditions, then the space  $V^{1,\phix}(\Omega)$ remains unchanged, up to equivalent norms, after replacing $\phi$ with the normalized generalized Young function $\widehat \phi$ given by \eqref{july11}, which plainly satisfies assumption \eqref{conv0}.

The Sobolev conjugate  of $\phi$ on  domains with finite measure can thus be defined as the generalized 
 Young function  $\phi_{n, \diamond}$ obeying 
\begin{equation}\label{sobconj1}
\phi_{n, \diamond} (x,t) = \widehat  \phi (x,H_{n, \diamond}^{-1}(x,t)) \qquad \text{for $x\in\RRn$ and $t \geq 0$,}
\end{equation}
where  $H_{n, \diamond}: \RRn \times [0, \infty) \to [0, \infty)$ is the function given by
\begin{equation}\label{H1}
H_{n, \diamond}(x,t) = \bigg(\int_0^t \bigg(\frac \tau{\widehat \phi(x,\tau)}\bigg)^{\frac 1{n-1}}\, d\tau\bigg)^{\frac {1}{n'}} \quad \text{for $x \in \RRn$ and $t \geq 0$.}
\end{equation}

  We begin with a Sobolev-Poincar\'e inequality for functions in $V^{1,\phix}_0(\Omega)$. This only requires that $\Omega$ has a finite measure.

\begin{theorem}{\rm\bf{ [Poincar\'e-Sobolev inequality with zero  boundary values]}}  \label{thm:mainzero} Let $\Omega$ be an open set in $\RRn$ such that $|\Omega|<\infty$ and let $\phi $ be a  generalized Young function satisfying conditions (A0) and (A1).  
Then, $$V^{1,\phi(\cdot)}_{0}(\Omega)\embedding L^{\phi_n(\cdot)}(\Omega),$$ and there exists a constant $c=c(\beta, n, |\Omega|)$  such that
\begin{equation}\label{main1zero}
\|u\|_{L^{\phi_{n, \diamond}(\cdot)}(\Omega)}\leq c \|\nabla u\|_{L^{\phi(\cdot)}(\Omega)}
\end{equation}
for every   $u \in V^{1,\phi(\cdot)}_0(\Omega)$. Here,  $\beta$ denotes the constant appearing in conditions (A0) and (A1).
\iffalse
\\  Moreover,  the Musielak-Orlicz space $L^{\widehat\phi_n(\cdot)}(\RRn)$ is optimal in \eqref{main1'}.
\fi
\end{theorem}

Like any other mean-value Sobolev-Poincar\'e inequality, this kind of inequalities in Musielak-Orlicz spaces requires that the domain $\Omega$ be regular enough.  Any bounded John domain is admissible. Recall that
an open set $\Omega \subset \RRn$ is called a bounded
 John domain if there exist positive constants $c$ and $\ell$, and a
point $x_0 \in \Omega$ such that for every $x \in \Omega$ there exists a
rectifiable curve $\gamma : [0, \ell] \to \Omega$, parametrized by
arclenght, such that $\gamma (0)=x$, $\gamma (\ell) = x_0$, and
$${\rm dist}\, (\gamma (r) , \partial \Omega ) \geq c r \qquad
\hbox{for $r \in [0, \ell]$.}$$
Any connected bounded open set
satisfying the uniform
interior cone condition is a bounded John domain. In particular, connected bounded Lipschitz domains are bounded John domains. A bounded John domain may admit inward cusps, whereas outward cusps are forbidden. The class of bounded John domains also includes certain domains with highly irregular boundaries. This is the case of the interior of  Koch’s snowflake, for instance.

\begin{theorem}{\rm\bf{ [Mean-value Poincar\'e-Sobolev inequality]}}  \label{thm:main'} Let $\Omega$ be a bounded John domain in $\RRn$ and let $\phi $ be a  generalized Young function satisfying conditions (A0) and (A1). 
Then, $$V^{1,\phi(\cdot)}(\Omega)\embedding L^{\phi_n(\cdot)}(\Omega),$$  and there exists a constant $c=c(\beta, \Omega)$  such that
\begin{equation}\label{main1'}
\|u-u_\Omega\|_{L^{\phi_{n, \diamond}(\cdot)}(\Omega)}\leq c \|\nabla u\|_{L^{\phi(\cdot)}(\Omega)}
\end{equation}
for every   $u \in V^{1,\phi(\cdot)}(\Omega)$. Here,  $\beta$ denotes the constant appearing in conditions (A0) and (A1).
\end{theorem}

{\color{black}
We conclude our discussion about Musielal-Orlicz-Sobolev embeddings with a condition for compactness.
%It extends a result for classical Orlicz-Sobolev spaces from \cite[Theorem 3.5]{cianchi_forum} (see also \cite[Theorem 3]{cianchi1} for an equivalent formulation).

\begin{theorem}{\rm\bf{ [Compact embeddings]}}  \label{thm:maincompact}  Let $\Omega$ be an open set in $\RRn$ such that $|\Omega|<\infty$ and let $\phi $ be a  generalized Young function satisfying conditions (A0) and (A1). Let $\vartheta$ be a generalized Young function growing essentially more slowly than $\phi_{n, \diamond}$ near infinity and such that
  \begin{align}\label{nov170}
      \int_\Omega \vartheta(x,t)\,dx < \infty \qquad \text{for $t>0$}.
  \end{align}
 (i) The embedding 
\begin{equation}\label{nov160}
 V^{1,\phix}_0(\Omega) \embedding L^{\vartheta (\cdot)}(\Omega)
\end{equation}
is compact.
\\ (ii) Assume, in addition, that $\Omega$ is a bounded 
John domain. Then, the embedding 
\begin{equation}\label{nov161}
 V^{1,\phix}(\Omega) \embedding L^{\vartheta (\cdot)}(\Omega)
\end{equation}
is compact.
\end{theorem}
}

\begin{remark}\label{special} {\rm
Observe that, if the function $\phi$ fulfills condition \eqref{conv0}, and hence the function $\phi_n$ is well defined, then $\phi_n \simeq \phi_{n,\diamond}$ near infinity. Therefore, since $|\Omega|<\infty$, inequalities  \eqref{main1zero} and \eqref{main1'} can be equivalently formulated with the norm $\|\cdot\|_{L^{\phi_{n, \diamond}(\cdot)}(\Omega)}$ replaced with $\|\cdot\|_{L^{\phi_{n}(\cdot)}(\Omega)}$. An equivalent inequality is also obtained whenever $\phi_{n,\diamond}$ is replaced with the function defined as in \eqref{sobconj1}--\eqref{H1}, with $\widehat \phi$ replaced with any generalized Young function $\psi$ such that $\psi(x,t) \simeq \widehat \phi (x,t)$ for $x \in \Omega$ and $t \geq 0$.
Moreover, under the additional assumption that $\phi$ satisfies the $\Delta_2$-condition, 
owing to equation \eqref{oct49} the norm $\|\cdot \|_{L^{\phi_{n, \diamond}(\cdot)}(\RRn)}$ in inequalities  \eqref{main1zero} and   \eqref{main1'} can be replaced with the functional obtained via the function $\phi^\bullet$ defined by \eqref{oct48}, instead of $\widehat \phi$, in \eqref{sobconj1} and \eqref{H1}.

}
\iffalse

{\color{orange}{\rm If the function $\phi$ satisfies property \eqref{sep62}, then normalization of $\phi$ is not needed in the construction of the function $\phi_n$ defining the target Musielak-Orlicz target space in the Sobolev-Poincar\'e inequality \eqref{main1'}. 
One can directly define the latter space via the 
 function $\phi_n^\bullet$ given  as in \eqref{sobconj}--\eqref{H}, with $\phi$ replaced with the function $\phi^\bullet$ defined by \eqref{sep60}.
Indeed, in Theorem \ref{thm:main'} one can normalize $\phi$  by choosing the function $\widehat \phi$ defined as in \eqref{july11}. 
Thanks to property \eqref{sep61}, the functions $\phi^\bullet_n\chi_\Omega$ and $\widehat \phi_n\chi\Omega$ are also equivalent, and hence the spaces $L^{\phi_n^\bullet(\cdot)}(\Omega)$ and $L^{\phi_n(\cdot)}(\Omega)$ agree, up to equivalent norms.
}}
\fi
\end{remark}

The remaining part of this section is devoted to applications of Theorems \ref{thm:main}, \ref{thm:mainzero} and \ref{thm:main'} in the case of Musielak-Orlicz-Sobolev spaces built upon generalized Young functions of  the special forms presented in Section \ref{background}.

\begin{example} {\rm\bf{[Classical Orlicz-Sobolev conjugate]}} \label{orliczsob}{\rm Assume that $\phi$ is independent of $x$, namely 
$$\phi (x,t) = A(t)$$
for some classical Young function $A$. By replacing, if necessary, $A$ with an equivalent Young function, we may assume, without loss of generality, that $A(1)=1$. Consequently, $\phi$ is normalized and $\phi = \overline \phi$. The Sobolev conjugate $\phi_n$ is hence also independent of $x$. Thus, with  abuse of notation, we may drop the dependence on $x$, and formulas \eqref{sobconj} and \eqref{H} yield
\begin{equation}\label{oct50}
\phi_n (t) =  A(H_n^{-1}(t)) \qquad \text{for  $t \geq 0$,}
\end{equation}
where  $H_n:  [0, \infty) \to [0, \infty)$ is the function given by
\begin{equation}\label{oct51}
H_n(t) = \bigg(\int_0^t \bigg(\frac \tau{A(\tau)}\bigg)^{\frac 1{n-1}}\, d\tau\bigg)^{\frac {1}{n'}} \quad \text{for   $t \geq 0$.}
\end{equation}
The function $\phi_n$ reproduces the sharp Sobolev conjugate of $A$, in the form exhibited in \cite{cianchi2}.
}
\end{example}

\begin{example} {\rm\bf{[Variable exponent Sobolev conjugate]}}   \label{variablesob}{\rm
Let $\phi$ be the variable exponent function as in Example \ref{exvariable}. Namely,
$$\phi(x,t) = t^{p(x)} \quad \text{for  $x\in \RRn$ and $t \geq0$.}$$
Clearly, 
$$\phi_\infty (t)= t^{p_\infty} \quad \text{for $t \geq 0$.}$$
 Consider inequalities in $\RRn$ and suppose that 
\begin{equation}\label{sep169}
p_\infty <n,
\end{equation} 
an assumption indispensable in view of  Proposition \ref{nec}.
% Since
%\begin{equation}\label{sep160}
%\phi_\infty (t) = t^{p_\infty} \qquad \text{if $0 \leq t <1$,}
%\end{equation}
 Since $p\in L^\infty(\RRn)$, the function $\phi$ satisfies the  $\Delta_2$-condition. Thus, in view of Remark  \ref{rem-oct}, using the function $\phi^\circ$, given by \eqref{july13}, instead of $\overline \phi$ in the definition of $\phi_n$, results in an equivalent function. Let us still denote, with abuse of notation,  by $H_n$ and $\phi_n$ the functions obtained after this replacement.
One has that
$$ \phi^\circ (x,t) = \begin{cases} t^{p(x)} & \quad \text{if $t \geq 1$}
\\ t^{p_\infty}& \quad \text{if $0 \leq t <1$,}
\end{cases}$$
for $x \in \RRn$.
%the function $\phi$ can be normalized by choosing the function $\phi^\diamond$ associated with $\phi$ as in  \eqref{july13}.
Thus, 
\begin{equation}\label{sep161}
H(x,t)=t_0\,t^{\frac{n-p_\infty}n} \qquad \text{if $0 \leq t <1$,}
\end{equation}
and 
\begin{equation}\label{oct10}
H^{-1}(x,t)=   t_0^{-\frac{n}{n-p_\infty}} t^\frac{n}{n-p_\infty}\qquad \text{if $t <t_0$,}
%\Big(\frac{n-p_\infty}{n-1}\Big)^{\frac{n-1}{n-p_\infty}} t^\frac{n}{n-p_\infty}\qquad \text{if $t < \big(\tfrac{n-1}{n-p_\infty}\big)^{\frac{1}{n'}}$,}
\end{equation}
where we have set
$t_0= \big(\tfrac{n-1}{n-p_\infty}\big)^{\frac{1}{n'}}$.
Consequently,
\begin{equation}\label{oct11}
\phi_n(x,t) = t_0^{-\frac{np_\infty}{n-p_\infty}} t^\frac{np_\infty}{n-p_\infty}\qquad \text{if $t <t_0$.}
%
%\phi(x, H^{-1}(x,t)) =  \Big(\frac{n-p_\infty}{n-1}\Big)^{\frac{(n-1)p_\infty}{n-p_\infty}} t^\frac{np_\infty}{n-p_\infty}\qquad \text{if $t < \big(\tfrac{n-1}{n-p_\infty}\big)^{\frac{1}{n'}}$.}
\end{equation}
The behavior of the function $\phi_n(x,t)$ for $t \geq t_0$ takes a different form, depending on whether $1 \leq p(x)<n$, $p(x)=n$, or $p(x)>n$.
\\  Assume first that $1 \leq p(x)<n$. Then,
\begin{equation}\label{sep162}
H(x,t)= \Big(\frac{n-1}{n-p(x)}\Big)^{\frac{1}{n'}}\Big( \frac{p_\infty - p(x)}{n-p_\infty} +t^{\frac{n-p(x)}{n-1}}\Big) ^{\frac{1}{n'}}\qquad \text{if $t \geq 1$,}
\end{equation}
whence 
\begin{equation}\label{sep163}
H^{-1}(x,t)= \bigg(\Big(\frac{n-p(x)}{n-1}\Big)^{\frac{1}{n'}}t^{n'} -   \frac{p_\infty - p(x)}{n-p_\infty} \bigg) ^{\frac{n-1}{n-p(x)}}\qquad \text{if $t \geq t_0$.}
\end{equation}
Therefore,
\begin{equation}\label{oct12}
\phi_n(x,t) =  \bigg(\Big(\frac{n-p(x)}{n-1}\Big)^{\frac{1}{n'}}t^{n'} -   \frac{p_\infty - p(x)}{n-p_\infty} \bigg) ^{\frac{(n-1)p(x)}{n-p(x)}}
\quad   \text{if $t \geq t_0$.} 
\end{equation}
In particular, 
\begin{equation}\label{sep164}
\phi_n(x,t)   \approx (n-p(x))^{\frac 1{n-p(x)}}t^{\frac{np(x)}{n-p(x)}} \quad \text{if $t \geq t_1$,}
\end{equation}
for a suitable $t_1> t_0$ and with equivalence constants independent of $x$.
\\ Next, assume that $p(x)=n$. Since
\begin{equation}\label{sep165}
H(x,t)=  \Big( \frac{n-1}{n-p_\infty} +\log t\Big) ^{\frac{1}{n'}}\qquad \text{if $t \geq 1$,}
\end{equation}
we obtain that
\begin{equation}\label{sep165b}
\phi_n(x,t)  =e^{\frac{n(1-n)}{n-p_\infty}}e^{nt^{n'}}\qquad \text{if $t \geq t_0$.}
 %\big( \tfrac{n-1}{n-p_\infty}\big) ^{\frac{1}{n'}}$.}
\end{equation}
 %equivalence constants independent of $x$.
\\ Finally, if $p(x)>n$, then
\begin{equation}\label{sep166}
H(x,t)= \Big(\frac{n-1}{p(x)-n}\Big)^{\frac{1}{n'}}\Big( \frac{p(x)-p_\infty}{n-p_\infty} -t^{\frac{n-p(x)}{n-1}}\Big) ^{\frac{1}{n'}}\qquad \text{if $t \geq 1$.}
\end{equation}
Thereby,
\begin{equation}\label{sep167}
H^{-1}(x,t)= \begin{cases} \infty &\quad \text{if $t \geq t_\infty(x)$}
\\
\displaystyle \bigg(\frac{p(x)-p_\infty}{n-p_\infty} -   \frac{p(x)-n}{n-1}t^{n'} \bigg) ^{\frac{n-1}{n-p(x)}}&\quad \text{if $ t_0 \leq t < t_\infty(x)$,}
\end{cases}
\end{equation}
where we have set $t_\infty(x) =\big( \frac{(p(x)-p_\infty)(n-1)}{(n-p_\infty)(p(x)-n)}\big)^{\frac{1}{n'}}$.
% a  number larger than $1$. 
Hence,
\begin{equation}\label{sep168}
\phi_n(x,t) = \begin{cases} \infty &\quad \text{if $t \geq t_\infty (x)$}
\\
\displaystyle \bigg(\frac{p(x)-p_\infty}{n-p_\infty} -   \frac{p(x)-n}{n-1}t^{n'} \bigg) ^{\frac{(n-1)p(x)}{n-p(x)}}&\quad \text{if $ t_0\leq t < t_\infty (x)$.}
\end{cases}
\end{equation}
Altogether,
by Theorem \ref{thm:main}, inequality \eqref{main1} follows with $\phi_n$
given by \eqref{oct11}, \eqref{oct12},

  As far as inequalities on domains are concerned,  inequalities \eqref{main1zero} and \eqref{main1'} hold with $\phi_{n, \diamond}$ equivalent, for large values of $t$, to the function $\phi_n$ defined as above, even if assumption \eqref{sep169} is dropped. These conclusions recover results from \cite{Diening04} and \cite{HarHas08Sobolev1n}, and also extend them to the case when $p(x)>n$.
}

\end{example}

\begin{example}{\rm\bf{[Double-phase Sobolev conjugate]}}\label{doublesob}  {\rm
Consider  the double-phase function   
$$\phi(x,t)= t^p + a(x)t^q \quad \text{for  $x\in \RRn$ and $t \geq0$,}$$
from Example \ref{doubleex}.
In particular, $1 \leq p < q$ and the function $a$ belongs to $L^\infty (\RRn)$. Hence, the function $\phi$  satisfies the $\Delta_2$-condition \eqref{delta2}.
Since $0\leq a(x) \leq \|a\|_{L^\infty (\RRn)}$, one can verify that
$$\phi (x, \phi^{-1}(x,1)t) \approx \phi(x,t) \quad \text{for $x \in \RRn$ and $t \geq 0$.}$$
Moreover, 
inasmuch as 
\begin{equation}\label{sep150}
t^p \leq \phi(x,t) \leq (1+\|a\|_{L^\infty (\RRn)})\, t^p \qquad \text{if $0\leq t \leq 1$,}
\end{equation}
for $x \in \RRn$,
we   have that
$$\phi_0(x,t) \approx \phi(x,t)\quad \text{for $x \in \RRn$ and $t \geq 0$,}$$
where $\phi_0$ is the function defined as in \eqref{july8}, and also
$$2\phi_0(x,t)-1 \approx \phi(x,t)\quad \text{if $t \geq 1$,}$$
for $x \in \RRn$. Altogether, we infer 
 that
$$\overline \phi (x,t) \approx \phi(x,t) \quad \text{for $x \in \RRn$ and $t \geq 0$}.$$
Let us focus on inequalities in $\RRn$.
By equation \eqref{sep150},
$$\phi_\infty(t) \approx t^p \qquad \text{if $0\leq t \leq 1$.}
$$
Hence, condition \eqref{conv0} entails that
$p<n$.
%with equivalence constants independent of $x$ and $t$.
%we have that
%\begin{equation}\label{sep151}
%\phi_\infty (t) \approx  \phi(x,t) \qquad \text{if $0\leq t \leq 1$,}
%\end{equation}
%for $x \in \RRn$, with equivalence constants independent of $x$ and $t$.
 %Thus,
%the function $\phi^\diamond$ associated with $\phi$ as in  \eqref{july13} is equivalent to $\phi$. 
In view of Remark \ref{rem-oct}, we can define the function $\phi_n$  via $\phi$ itself instead of $\overline \phi$. 
\\ If  $a(x)=0$, then $\phi(x,t)= t^p$, whence 
$$\phi_n(x,t) \approx t^{\frac {np}{n-p}} \quad \text{for $t \geq 0$.}$$
Let us consider the nontrivial case when $a(x)\neq 0$.
 Therefore,  
\begin{align}\label{sep152}
H_n(x,t) =
 \bigg(\int_0^t \bigg(\frac \tau{ \tau^p + a(x)\tau^q  }\bigg)^{\frac 1{n-1}}\, d\tau\bigg)^{\frac {1}{n'}} 
 = 
a(x)^{\frac{n-p}{(p-q)n}} \bigg(\int_0^{ta(x)^{\frac 1{q-p}}} \frac {ds}{ (s^{p-1} + s^{q-1})^{\frac 1{n-1}}  }\,\bigg)^{\frac {1}{n'}} 
\end{align}
for $x \in \RRn$ and $t \geq 0$. Define the function $F: [0, \infty) \to [0, \infty)$ as 
\begin{equation}\label{F} F(r) = \bigg(\int_0^{r} \frac {ds}{ (s^{p-1} + s^{q-1})^{\frac 1{n-1}}  }\,\bigg)^{\frac {1}{n'}} \quad \text{for $r \geq 0$.}
\end{equation}
Hence,
\begin{equation}\label{sep154}
H^{-1}_n(x,t) = a(x)^{\frac 1{p-q}} F^{-1}\big(t a(x)^{\frac{n-p}{n(q-p)}}\big)\qquad \text{for $x \in \RRn$ and $t \geq 0$.}
\end{equation}
Assume first that $q <n$. Then the function $F$
 obeys
\begin{equation}\label{sep153}
F^{-1}(t) \approx \begin{cases} t^\frac{n}{n-q} & \qquad \text{if $t \geq 1$}
\\ t^\frac n{n-p}  & \quad \text{if $0 \leq t <1$,}
\end{cases}
\end{equation}
with equivalence constants depending on $p$, $q$ and $n$.
From equations \eqref{sep154} and  \eqref{sep153}, one can deduce   that
\begin{equation}\label{sep155}
\phi_n(x,t) = \phi(x, H^{-1}_n(x,t)) \approx t^{\frac{np}{n-p}} + a(x)^{\frac n{n-q}} t^{\frac{nq}{n-q}} \quad \text{for $x \in \RRn$ and $t \geq 0$.}
\end{equation}
This recovers \cite[Proposition 3.4]{Ho23} in the case of constant exponents.
\\ Assume next that $q=n$. Let $F$ be defined as in \eqref{F}, with $q=n$. Then, 
\begin{equation}\label{sep156}
F^{-1}(t) \approx \begin{cases} e^{t^{n'}} & \quad \text{if $t \geq 1$}
\\ t^\frac n{n-p}  & \quad \text{if $0 \leq t <1$.}
\end{cases}
\end{equation}
Hence, 
\begin{equation}\label{sep157}
H^{-1}_n(x,t) = a(x)^{\frac 1{p-n}} F^{-1}\big(t a(x)^{\frac{1}{n}}\big)\approx 
\begin{cases}  a(x)^\frac 1{p-n} e^{a(x)^\frac 1{n-1} t^{n'}} & \quad \text{if $t a(x)^{\frac{1}{n}} \geq 1$}
		\\
 t^\frac n{n-p} & \quad \text{if $t a(x)^{\frac{1}{n}} <1$,}
\end{cases}
\end{equation}
and 
\begin{equation}\label{sep158}
\phi_n(x,t) = \phi(x, H^{-1}_n(x,t)) \approx 
\begin{cases}  a(x)^\frac p{p-n} e^{na(x)^\frac 1{n-1} t^{n'}} & \quad \text{if $t a(x)^{\frac{1}{n}} \geq 1$}
		\\
 t^\frac {np}{n-p} & \quad \text{if $t a(x)^{\frac{1}{n}} <1$.}
\end{cases}
\end{equation}
Therefore,
\begin{equation}\label{sep158b}
\phi_n(x,t)
 \simeq  t^\frac {np}{n-p} + a(x)^\frac p{p-n} \bigg(e^{a(x)^\frac 1{n-1} t^{n'}} - \sum_{k=0}^{\lceil\frac{p(n-1)}{n(n-p)}\rceil}\frac{(a(x)^\frac 1n t)^{n'k}}{k!}\bigg) \quad \text{for $x \in \RRn$ and $t \geq 0$,}
\end{equation}
where $\lceil\cdot\rceil$ stands for integer part.
\\
 Finally, suppose that $q>n$. Set 
$$t_\infty = \bigg(\int_0^{\infty} \frac {ds}{ (s^{p-1} + s^{q-1})^{\frac 1{n-1}}  }\,\bigg)^{\frac {1}{n'}} .$$
 One can verify that
\begin{equation}\label{sep170}
F^{-1}(t) \approx \begin{cases} \infty & \quad \text{if $t \geq t_\infty$}
\\  (t_\infty -t)^{-\frac {n-1}{q-n}} & \quad \text{if $t_\infty/2 \leq t  <t_\infty$}
\\ t^\frac n{n-p}  & \quad \text{if $0 \leq t <t_\infty/2$}
\end{cases}
\approx  \
 \begin{cases} \infty & \quad \text{if $t \geq t_\infty$}
\\  t^\frac n{n-p} (t_\infty -t)^{-\frac {n-1}{q-n}} & \quad \text{if $0\leq t  <t_\infty$.}
\end{cases}
\end{equation}
Hence, by equation \eqref{sep154},
\begin{align}\label{sep171}
H^{-1}(x,t) & \approx \begin{cases} \infty & \quad \text{if $t a(x)^{\frac{n-p}{n(q-p)}} \geq t_\infty$}
\\  a^{\frac 1{p-q}}\Big(t a(x)^{\frac{n-p}{n(q-p)}}\Big)^{\frac{n}{n-p}} (t_\infty -t a(x)^{\frac{n-p}{n(q-p)}})^{-\frac {n-1}{q-n}} & \quad \text{if $0 \leq t a(x)^{\frac{n-p}{n(q-p)}}   <t_\infty$}
\end{cases}
\\ \nonumber& 
\\ \nonumber &= \begin{cases} \infty & \quad\quad\quad\quad \quad\quad\quad\quad\quad \text{if $t a(x)^{\frac{n-p}{n(q-p)}} \geq t_\infty$}
\\  t^{\frac{n}{n-p}} (t_\infty -t a(x)^{\frac{n-p}{n(q-p)}})^{-\frac {n-1}{q-n}} &\quad\quad\quad\quad\quad\quad\quad\quad \quad \text{if $0 \leq t a(x)^{\frac{n-p}{n(q-p)}}   <t_\infty$.}
\end{cases}
\end{align}
Thereby, 
\begin{align}\label{sep172}
\phi_n(x,t) =\phi(x,H^{-1}(x,t)) & \approx   \begin{cases} \infty & \quad \text{if $t  \geq t_\infty a(x)^{\frac{n-p}{n(p-q)}}$}
\\ \\  \displaystyle
\frac{a^{\frac p{p-q}}\Big(t a(x)^{\frac{n-p}{n(q-p)}}\Big)^{\frac{np}{n-p}}}{ (t_\infty -t a(x)^{\frac{n-p}{n(q-p)}})^{\frac {(n-1)p}{q-n}} }
 + \frac{a^{\frac p{p-q}}\Big(t a(x)^{\frac{n-p}{n(q-p)}}\Big)^{\frac{nq}{n-p}}}{ (t_\infty -t a(x)^{\frac{n-p}{n(q-p)}})^{\frac {(n-1)q}{q-n}} } 
 & \quad \text{if $0 \leq t     <t_\infty a(x)^{\frac{n-p}{n(p-q)}}$.}
\end{cases}
\\ \nonumber &
\\ \nonumber &
 \approx 
  \begin{cases} \infty & \quad  \quad\quad\quad\quad \quad\quad\quad\quad \quad\quad\quad\quad\text{if $t   \geq t_\infty a(x)^{\frac{n-p}{n(p-q)}}$}
\\   
\displaystyle \frac{t ^{\frac{np}{n-p}}}{ (t_\infty -t a(x)^{\frac{n-p}{n(q-p)}})^{\frac {(n-1)q}{q-n}} }
 & \quad  \quad\quad\quad\quad \quad\quad\quad\quad \quad\quad\quad\quad \text{if $0 \leq t    <t_\infty a(x)^{\frac{n-p}{n(p-q)}}$.}
\end{cases}
\end{align}
}

\end{example}

\section{Weak-type estimates for Riesz potentials}\label{sec:riesz}

This section is devoted to a weak-type inequality for Riesz potentials in Musielak-Orlicz spaces, which plays a crucial role in the proof of our Sobolev inequalities.

Given $\alpha \in (0,n)$, we denote by $I_\alpha$ the Riesz potential operator, defined as 
\begin{equation}\label{riesz}
I_\alpha f (x) = \int_{\RRn}\frac{f(y)}{|x-y|^{n-\alpha}}\,dy \qquad \text{for $x\in \RRn$,}
\end{equation}
for $f \in L^0(\RRn)$.

%
%Like in Subsection \ref{Rn}, here we deal with generalized  Musielak-Orlicz spaces built upon Young functions $\phi$ satisfying 
%properties (A0), (A1), and (A2). Owing to Proposition \ref{replacephi},  we shall suppose, without loss of generality, that $\phi$ is normalized. 

The  definition of the generalized Young function  $\phi_{\frac n\alpha}$ that enters the weak-type inequality for $I_\alpha$ in the space $L^{\phix}(\RRn)$ extends that of $\phi_n$  to all  $\alpha \in (0,n)$. 
Assume that $\phi$ is a generalized Young function satisfying conditions (A0), (A1) and (A2), and such that
\begin{equation}\label{conv0a}
\int_0 \bigg(\frac t{\phi_\infty(t)}\bigg)^{\frac \alpha{n-\alpha}}\, dt <\infty,
\end{equation}
a condition which will be shown to be necessary for any weak-type inequality for  $I_\alpha$  to hold in  $L^{\phix}(\RRn)$.
%Let $\overline \phi$ be the normalized generalized Young function associated with $\phi$ as in  \ref{july7}. 
\\ The Sobolev conjugate of order $\alpha$ of $\phi$ 
is defined 
as 
\begin{equation}\label{sobconja}
\phi_{\frac n\alpha}  (x,t) =  \overline \phi\big (x,H_{\frac n\alpha} ^{-1}(x,t)\big) \qquad \text{for $x\in\RRn$ and $t \geq 0$,}
\end{equation}
where $\overline \phi$ is the normalized version of $\phi$ given by \eqref{july7}, and the function $H_{\frac n\alpha} : \RRn \times [0, \infty) \to [0, \infty)$ obeys
\begin{equation}\label{Ha}
H_{\frac n\alpha} (x,t) = \bigg(\int_0^t \bigg(\frac \tau{ \overline \phi(x,\tau)}\bigg)^{\frac \alpha{n-\alpha}}\, d\tau\bigg)^{\frac {n-\alpha}n} \quad \text{for $x \in \RRn$ and $t \geq 0$.}
\end{equation}
Owing to  assumption \eqref{conv0a} and property \eqref{july5} of normalized generalized Young functions,   the function $H$ is finite-valued for $x \in \RRn$ and $t \geq 0$. The definition of the inverse  $H_{\frac n\alpha} ^{-1}(x,\cdot)$ is analogous to that of $H_{n} ^{-1}(x,\cdot)$,  explained in the previous section. Plainly, choosing $\alpha =1$ in \eqref{sobconja} and \eqref{Ha} reproduces the function $\phi_n$ given by \eqref{sobconj}.

\begin{theorem} {\rm\bf{[Weak-type inequality for Riesz potentials]}} \label{lemma12} Let $\alpha \in (0,n)$ and let $\phi $ be a  generalized Young function satisfying 
conditions
(A0), (A1), (A2), and 
\eqref{conv0a}. 
Then, there exists a positive constant  $c=c(n, \alpha, \beta, h)$  such that
\begin{equation}\label{176'}
\int_{\{|I_\alpha f|>t\}}\phi_{\frac n\alpha}(x,ct)\, dx \leq \int_{\RRn}\overline \phi(x, |f(x)|)\, dx \quad \text{for $t \geq 0$,}
\end{equation}
if $\|f\|_{L^{\phi(\cdot)}(\RRn)}\leq 1$. Here,  $\beta$ denotes the constant from  conditions (A0) and (A1), and $h$ the function from condition (A2). In particular, if $\phi= \overline \phi$, then the constant $c$ in inequality \eqref{176'} only depends on $n$, $\alpha$, and $\beta$.
\iffalse
\\ Moreover, inequality \eqref{176'} is sharp, in the sense that if it holds with   $\phi_{\frac n\alpha}$ replaced by another generalized $\Phi$-function $\vartheta$ satisfying properties (A0), (A1) and (A2), then there exists a constant $c'$ such that
\begin{equation}\label{july20}
\vartheta(x,t) \leq \phi_{\frac n\alpha}(x,c't) \qquad \text{for a.e. $x \in \RRn$ and for every $t\geq 0$.}
\end{equation}
\fi
\end{theorem}

\begin{remark}\label{rem-oct5}{\rm It will be clear from the proof that inequality \eqref{176'} also holds if $\overline \phi$ is replaced with any normalized generalized Young function satisfying condition \eqref{conv0a}, and $\phi_{\frac n\alpha}$ is  directly defined by the use of $\phi$, instead of $\overline \phi$, in \eqref{sobconja} and \eqref{Ha}.}
\end{remark}

The necessity of assumption \eqref{conv0a} in Theorem \ref{lemma12} is shown by the next proposition. 

\begin{proposition}{\rm\bf{ [Necessity of condition \eqref{conv0a}]}}\label{optweak}
Let $\alpha$ and  $\phi $   be as in Theorem \ref{lemma12}.
% be a generalized  Young function satisfying properties (A0), (A1), and (A2).  
Assume that there exist a generalized Young function $\vartheta$ and a positive constant $c$ such that 
\begin{equation}\label{optweak1}
\int_{\{|I_\alpha f|>t\}}\vartheta (x,ct)\, dx \leq \int_{\RRn}\overline \phi(x, |f(x)|)\, dx \quad \text{for $t \geq 0$,}
\end{equation}
if $\|f\|_{L^{\phi(\cdot)}(\RRn)}\leq 1$. Then, condition \eqref{conv0a} must be fulfilled.
\end{proposition}

For convenience, we split the main steps of the proof of Theorem \ref{lemma12} into separate lemmas.
Unless otherwise stated, in these lemmas and the in whole remaining part of this section, $\phi$ denotes a normalized generalized Young function fulfilling condition \eqref{conv0a},
and $\beta$  the constant from Definition \ref{normalized}.  The lemmas will then be applied with  $\phi = \overline \phi$  in accomplishing the proof of Theorem \ref{lemma12}.

We begin by introducing some notations and definitions.
Let $k\in \RR$ be such that
\begin{equation}\label{k1}
k\geq \frac 4\beta,
\end{equation}
and let $\sigma \in \RR$ be such that
\begin{align}\label{sigmanew}
\sigma \geq \omega_n \max\Big\{\frac{1}{2^n}, \frac n{n-\alpha}\Big\},
\end{align}
where $\omega_n$ denotes the Lebesgue measure of the unit ball in $\RRn$.
Define the generalized Young function $\psi$ as
\begin{equation}\label{30}
\psi (x,t) = \sigma \, t^{\frac n{n-\alpha}}\int_0^t\frac{\widetilde \phi (x, k\tau)}{\tau ^{1+\frac n{n-\alpha}}}\, d\tau \quad \text{for $x\in \RRn$ and $t \geq 0$,}
\end{equation}
 and the function $\lambda : \RRn \times [0, \infty) \to [0, \infty]$ as
\begin{equation}\label{31}
\lambda (x,\delta) = \frac 1{\delta ^{n-\alpha}\psi^{-1}(x, \delta^{-n})} \quad \text{if $\delta >0$,}
\end{equation}
and $\lambda (x,0)= \lim_{\delta \to 0^+} \lambda (x,\delta)$, for $x\in \RRn$. Notice that the latter limit exists and equals  $\Big(\sigma  \int_0^\infty\frac{\widetilde \phi (x, k\tau)}{\tau ^{1+\frac n{n-\alpha}}}\, d\tau\Big)^{\frac {n-\alpha}n}$.

In what follows, we shall   make use of the fact that, if $h: [0, \infty) \to [0, \infty)$ is a concave function such that $h(0)=0$, then
\begin{equation}\label{5}
\min\{1, c\}h(t) \leq h(ct) \leq \max\{1, c\}h(t) \quad\text{for  $c, t \geq 0$.}
\end{equation}

\begin{lemma}\label{lemma1} Let $x\in \RRn$ and let $\delta >0$. If 
\begin{equation}\label{33}
0\leq t \leq \frac 1{\lambda (x,\delta) \delta^{n-\alpha}},
\end{equation}
then
\begin{equation}\label{32}
 \widetilde \phi \Big(x+\frac{\nu}{(\lambda (x,\delta)t)^{\frac 1{n-\alpha}}}, t\Big)
 \leq \widetilde \phi\Big(x,\frac {2t}\beta\Big)
 \quad \text{for $\nu \in \mathbb S^{n-1}$.}
\end{equation}
\end{lemma}
\begin{proof} Fix  $x\in \RRn$ and  $\delta >0$, and, for ease of notation, denote $\lambda (x,\delta)$ simply by $\lambda$ throughout this proof. Since
\begin{equation}\label{34}
\frac 1{\lambda \delta^{n-\alpha}} = \psi ^{-1}(x, \delta^{-n}),
\end{equation}
inequality  \eqref{33} implies that
\begin{equation}\label{35}
t\leq \psi^{-1}(x, \delta^{-n}).
\end{equation}
Thus, owing to property \eqref{25},  inequality in \eqref{32} will follow if we show that
\begin{align}\label{36}
t \leq \frac \beta 2\widetilde \phi^{-1} \Big(x, \frac{2^n (\lambda t)^{\frac{n}{n-\alpha}}}{\omega_n}\Big)
\end{align}
for $t$ as in \eqref{35}.  
%\\ To this purpose, note that inequality \eqref{36} is equivalent to 
%\begin{align}\label{37}
%\widetilde \phi \big(x, \tfrac 2\beta t\big) \leq \frac{2^n (\lambda t)^{\frac{n}{n-\alpha}}}{\omega_n} = \frac {2^n}{\omega_n}\frac{(\beta \lambda)^{\frac n{n-\alpha}}}{2^{\frac n{n-\alpha}}} \Big(\frac {2t}\beta\Big)^{\frac n{n-\alpha}}.
%\end{align}
We have that 
\begin{align}\label{38}
\psi(x,t) &\geq \sigma t^{\frac n{n-\alpha}}\int_{\frac t2}^t \frac{\widetilde \phi(x, k\tau)}{\tau^{1+\frac n{n-\alpha}}}\, d\tau  \geq \sigma \widetilde \phi (x, \tfrac k2 t) t^{\frac n{n-\alpha}}\int_{\frac t2}^t \frac{1}{\tau^{1+\frac n{n-\alpha}}}\, d\tau 
\\ \nonumber & = \sigma \tp (x, \tfrac k2 t) \frac{2^{\frac n{n-\alpha}}-1}{\frac n {n-\alpha}} \geq \sigma  \tp (x, \tfrac k2 t).
\end{align}
Hence, by inequalities \eqref{35} and \eqref{k1},
%if 
%\begin{equation}\label{k1}
%k\geq \tfrac 4\beta,
%\end{equation}
%then
\begin{align}\label{39}
\tp (x, \tfrac 2\beta t)& \leq \frac 1\sigma \psi (x, \tfrac 4{k\beta}t) = \Big(\frac {4t}{k\beta} \Big)^{\frac n{n-\alpha}}\int_0^{\frac {4t}{k\beta}} \frac{\widetilde \phi(x, k\tau)}{\tau^{1+\frac n{n-\alpha}}}\, d\tau
\leq  \Big(\frac {4}{k\beta} \Big)^{\frac n{n-\alpha}} t^{\frac n{n-\alpha}} \int_0^{\frac {4}{k\beta}\psi^{-1}(x,\delta^{-n})} \frac{\widetilde \phi(x, k\tau)}{\tau^{1+\frac n{n-\alpha}}}\, d\tau
\\ \nonumber &  \leq    t^{\frac n{n-\alpha}} \int_0^{\psi^{-1}(x,\delta^{-n})} \frac{\widetilde \phi(x, k\tau)}{\tau^{1+\frac n{n-\alpha}}}\, d\tau
  = \frac{t^{\frac n{n-\alpha}}}{\sigma \delta^n \psi^{-1}(x, \delta^{-n})^{\frac n{n-\alpha}}}= \frac 1\sigma (\lambda t)^{\frac n{n-\alpha}}.
\end{align}
Thus,
%\todo[inline]{A: $\widetilde \phi^{-1}$ has to be right-continuous here}
\begin{align}\label{40}
t \leq \frac \beta 2 \tp^{-1}\Big(x,\frac{(\lambda t)^{\frac n{n-\alpha}}}\sigma\Big).
\end{align}
Coupling the latter inequality with \eqref{sigmanew} yields
\eqref{36}.
\end{proof}

%In what follows, $\chi_E$ denotes the characteristic function of a set $E \subset \RRn$.

In what follows, we denote by $B(x,r)$ the ball in $\RRn$ centered at $x\in \RRn$, with radius $r>0$.

\begin{lemma}\label{lemma2} 
% Let $\sigma $ and $k$ be as in \eqref{sigma1} and \eqref{k1}. Assume in addition that
%\begin{equation}\label{sigma2}
%\sigma \geq \frac {n\omega_n}{ n-\alpha}.
%\end{equation} 
Let   $x \in \RRn$ and $\delta \geq 0$. Then,
%Then {\color{yellow} there exists a constant $c=(n, \alpha, \beta, k)$ such that}
\begin{equation}\label{49}
%{\color{yellow} \frac c{\delta^{n-\alpha}\psi^{-1}(x,\delta^{-n})} \leq}
\bigg\|\frac{\chi_{\RRn \setminus B(x,\delta)}(\cdot)}{|x-\cdot|^{n-\alpha}}\bigg\|_{L^{\tp (\cdot)}(\RR)}\leq \lambda (x, \delta).
%\frac 1{\delta^{n-\alpha}\psi^{-1}(x,\delta^{-n})} 
\end{equation} 
\end{lemma}

\begin{proof} Fix $x \in \RRn$. Recall that
\begin{align}\label{50} 
\bigg\|\frac{\chi_{\RRn \setminus B(x,\delta)}(\cdot)}{|x-\cdot|^{n-\alpha}}\bigg\|_{L^{\tp (\cdot)}(\RRn)} = \inf \bigg\{ \lambda>0: \int_ {\RRn \setminus B(x,\delta)} \tp \Big(y, \frac 1{\lambda |x-y|^{n-\alpha}}\Big)\, dy\leq 1\bigg\}.
\end{align}
Choose  $\lambda = \lambda (x,\delta)$ given by \eqref{31}. Assume first that $\delta>0$, and hence $\lambda (x,\delta)<\infty$.
 By inequality \eqref{32} and  assumptions \eqref{k1} and  \eqref{sigmanew},
\begin{align}\label{51}
 \int_ {\RRn \setminus B(x,\delta)}& \tp \Big(y, \frac 1{\lambda |x-y|^{n-\alpha}}\Big)\, dy = \int_\delta^\infty \int_{\mathbb S^{n-1}}\tp\Big(x+\nu r, \frac 1{\lambda r^{n-\alpha}}\Big) r^{n-1}\, d\mathcal H^{n-1}(\nu)\, dr
\\ \nonumber & = \frac 1{(n-\alpha)\lambda ^{\frac n{n-\alpha}}}
\int_0^{\frac 1{\lambda \delta^{n-\alpha}}}
\int_{\mathbb S^{n-1}}\tp\Big(x+\frac{\nu}{(\lambda t)^{\frac 1{n-\alpha}}}, t\Big)\, d\mathcal H^{n-1}(\nu)\, \frac{dt}{t^{1+\frac n{n-\alpha}}}
\\ \nonumber & \leq \frac 1{(n-\alpha)\lambda ^{\frac n{n-\alpha}}} \int_0^{\frac 1{\lambda \delta^{n-\alpha}}}
\int_{\mathbb S^{n-1}}\tp\Big(x, \tfrac 2 \beta t\Big)\, d\mathcal H^{n-1}(\nu)\, \frac{dt}{t^{1+\frac n{n-\alpha}}}
\\ \nonumber & =\frac {n\omega_n}{(n-\alpha)\lambda ^{\frac n{n-\alpha}}} \int_0^{\frac 1{\lambda \delta^{n-\alpha}}}
  \frac{\tp\big(x, \tfrac 2 \beta t\big)}{t^{1+\frac n{n-\alpha}}}\, dt 
\\ \nonumber & \leq \frac {n\omega_n}{(n-\alpha)\lambda ^{\frac n{n-\alpha}}} \int_0^{\frac 1{\lambda \delta^{n-\alpha}}}
  \frac{\tp\big(x, k t\big)}{t^{1+\frac n{n-\alpha}}}\, dt.
\end{align} 
Here, $\mathcal H^{n-1}$ denotes the $(n-1)$-dimensional Hausdorff measure.
Since 
\begin{equation}\label{oct20}
\frac {n\omega_n}{(n-\alpha)\lambda ^{\frac n{n-\alpha}}} \int_0^{\frac 1{\lambda \delta^{n-\alpha}}}
  \frac{\tp\big(x, k t\big)}{t^{1+\frac n{n-\alpha}}}\, dt  =\frac {n\omega_n}{\sigma (n-\alpha)} \delta ^n \psi \Big(x, \frac 1{\lambda \delta^{n-\alpha}}\big)= 
\frac {n\omega_n}{\sigma (n-\alpha)} \leq 1,
\end{equation}
inequality \eqref{49} follows from \eqref{51} and \eqref{oct20}. 
\\ If 
$\delta =0$ and $\lambda (x,0)<\infty$, then 
\begin{equation} \label{oct21} 
\frac {n\omega_n}{(n-\alpha)\lambda ^{\frac n{n-\alpha}}} \int_0^{\frac 1{\lambda \delta^{n-\alpha}}}
  \frac{\tp\big(x, k t\big)}{t^{1+\frac n{n-\alpha}}}\, dt  = 
\frac {n\omega_n}{\sigma (n-\alpha)} \leq 1,
\end{equation}
%\todo[inline]{A: in the last equality we need $\psi (\psi^{-1}(t))\leq 1$. This is true, with $=$, with any definition of $\psi^{-1}$}
and inequality  \eqref{49} is a consequence of \eqref{51} and \eqref{oct21}. Finally, if  $\lambda (x,0)=\infty$, then inequality \eqref{49} holds trivially.
\end{proof}

\begin{lemma}\label{lemma3} Let $B$ be any ball in  $\RRn$ and let $x\in B$. Then,
%There exists  a constant $c=c(n)$ such that
% $c_1=c_1(n)$ and $c_2=c_2(n)$ such that
%\todo[inline]{A: first inequality to be checked}
\begin{equation}\label{80}
%{\color{yellow} \frac {c_1}{\beta \phi^{-1}(x, |B|^{-1}) } \leq} 
\|\chi_{B}\|_{L^{\phi (\cdot)}(\RRn)} \leq \frac {1}{\beta \phi^{-1}(x, |B|^{-1}) } 
\end{equation}
and
\begin{equation}\label{80'}
\|\chi_{B}\|_{L^{\tp (\cdot)}(\RRn)} \leq \frac 2{\beta \tp^{-1}(x, |B|^{-1}) }.
\end{equation}
\end{lemma}

Inequality \eqref{80} is proved in \cite[Proposition 4.7]{HarHas17riesz}, as a consequence of property \eqref{july4}. Inequality \eqref{80'} follows analogously, thanks to property \eqref{feb3}.
%
%In the statement of that proposition, it is assumed that $|B|\leq 1$. Such an assumption can be dropped under the assumptions of Lemma \ref{lemma3} since the function $\phi$ satisfies condition \eqref{Hc} for every $t\in [0, |B|^{-1}]$.
%\\ Equation \eqref{80'} holds thanks to \cite[..]{???}, since the function $\widetilde \phi$ fuflills 
%condition \eqref{feb3} and ???.

\begin{lemma}\label{lemma4}  There exists a constant $c=c(n, \alpha, \beta)$ such that
\begin{align}\label{85}
\delta^\alpha \leq 
%c \frac{1}{\phi(x, \delta^{-n})}\frac 1{\delta^{n-\alpha}\psi^{-1}(x, \delta^{-n})}= 
c \frac{\lambda(x,\delta)}{\phi^{-1}(x, \delta^{-n})} \quad \text{for $x \in \RRn$ and $\delta >0$.}
\end{align}
\end{lemma} 
\begin{proof} Fix $x \in \RRn$. By an elementary computation using polar coordinates and inequality \eqref{holder},
\begin{align} \delta^\alpha %= (\delta ^n)^{\frac \alpha n}
 \leq c
\int_{\RRn \setminus B(x,\delta)}
\frac{\chi_{B(x,2\delta)}(y)}{|x-y|^{n-\alpha}}\,dy 
  \leq 2c \|\chi_{ B(x,2\delta)}(\cdot)\|_{L^{\phi (\cdot)}(\RRn)} \bigg\|\frac{\chi_{\RRn \setminus B(x,\delta)}(\cdot)}{|x-\cdot|^{n-\alpha}}\bigg\|_{L^{\tp (\cdot)}(\RRn)}
\end{align}
for some constant $c=c(n, \alpha)$.
Inequality \eqref{85} hence follows, via \eqref{49} and \eqref{80}. Here, one has also to make use of property \eqref{5}.
\end{proof}

 Given a ball $B \subset \RRn$ and a function $f \in L^0(\RRn)$, we set
$$M_B f =  \dashint_B|f(y)|\,dy,$$
where $\dashint$ denotes a mean-value integral.
 The centered Hardy-Littlewood maximal function $Mf\,:\, \RRn \to [0,\infty]$ is then defined as
\begin{align*}
  Mf(x) =\sup_{r>0}   M_{B(x,r)} f \qquad \text{for $x\in \RRn$.}%  \dashint_{B_r(x)} \abs{f(y)}\,dy,
\end{align*}
Thanks to inequalities \eqref{holder} and \eqref{80'}, we have that $L^\phix (\RRn) \subset L^1_{\rm loc}(\RRn)$. Therefore, 
 $$Mf(x)<\infty \quad \text{ for a.e. $x\in \RRn$,}$$
 if  $f \in L^\phix (\RRn)$.
 
The following property of the averaging operator $M_B$ is a 
 consequence of \cite[Theorem~4.3.3]{HarHas19book}.

\begin{lemma}{\rm\bf{\cite{HarHas19book}}}     
  \label{pro:maxfunction}
  %Let $\phi$ be a normalized generalized Young function. Then,
There exists a positive constant  $\gamma=\gamma (n,\beta)$ such that 
    \begin{align}\label{average}
      \phi\big(x, \gamma M_B f  \big)
      &\leq   M_B\big( \phi(\cdot,f(\cdot))\big) % \dashint_B \phi(y,\abs{f(y)})\,dy
    \end{align}
for every ball $B \subset \RRn$, every $x \in B$ and every $f \in L^\phix(\RRn)$ with $\norm{f}_{\phi(\cdot)} \leq 1$.
%
%$\phi \in \Phi_1(\RRn)$ (or $\phi \in \Phi(\RRn)$ satisfying (A0), (A1) and (A2) with $h=0$). Then the following holds
%  \begin{enumerate}
%  \item [(i)]\label{itm:maxfunction-strong} There exists~$\gamma>0$ such that 
%    \begin{align*}
%      \phi\big(x, \gamma M_B f  \big)
%      &\leq   M_B\big( \phi(\cdot,f(\cdot)\big) % \dashint_B \phi(y,\abs{f(y)})\,dy
%    \end{align*}
%for every ball $B \subset \RRn$, every $x \in B$ and every $f \in L^\phix(\RRn)$ with $\norm{f}_{\phi(\cdot)} \leq 1$.
%  \item [(ii)] \label{itm:maxfunction-weak} 
%    There exists~$b>0$ such that 
%    \begin{align*}
%\int_{\{Mf>t\}}\phi(x,b t)\,dx \leq \int_{\RRn}\phi(x, |f(x)|)\,dx \quad \text{for   $t\geq 0$,}
%%
%%      \rho_\phi(\beta \lambda \indicator_{\set{Mf > \lambda}}) &\leq \rho_\phi(f).
%    \end{align*}
%for every $f \in L^\phix(\RRn)$ with $\norm{f}_{L^{\phix}} \leq 1$.  
%  \end{enumerate}
\end{lemma}
%\begin{proof}
%  This is a direct consequence of Theorem~4.3.3 and Lemma~4.3.11 of~\cite{HarHas19book}.
%\todo[inline]{A: shall we give some details?}
%\end{proof}

Define the function
%\begin{equation}\label{90}
%\Phi (x,t) = \frac 1{t^{\frac {n-\alpha}n}\psi^{-1}(x, t^{-1})}
%\end{equation}
%and
%\begin{equation}\label{91}
%\omega (x,t) = \Phi\Big(x, \frac{1}{\phi(x,t)}\Big).
%\end{equation}
%Notice that
%\begin{equation}\label{92}
%\lambda (x,t) = \Phi(x,\delta^n)
%\end{equation}
%and
$\omega : \RRn \times [0, \infty) \to (0, \infty]$ as
\begin{equation}\label{93}
\omega (x,t) = \lambda \Big(x, \frac 1{\phi (x,t)^{1/n}}\Big) \quad \text{for $x \in \RRn$ and $t \geq 0$,}
\end{equation}
where $\lambda$ is  given by \eqref{31}.

%The following estimate is well known -- see e.g. \cite[Inequality (3.1.1)]{adamshedberg}.
%\begin{lemma}\label{lemma6}
%We have that
%\begin{equation}\label{94}
%\int_{B(x,\delta)}\frac{|f(y)|}{|x-y|^{n-\alpha}}\, dy \leq \frac n\alpha \delta^\alpha Mf(x) \quad \text{for $x\in \RRn$.}
%\end{equation}
%\end{lemma}

\begin{lemma}\label{lemma7} There exists a constant $c=c(n, \alpha, \beta)$ such that
\begin{equation}\label{100}
|I_\alpha f(x)| \leq c \|f\|_{L^{\phi (\cdot)}(\RRn)}\, \omega \bigg(x, \frac{Mf(x)}{ \|f\|_{L^{\phi (\cdot)}(\RRn)}}\bigg) \quad \text{for a.e. $x\in \RRn$,}
\end{equation}
for $f \in L^{\phi (\cdot)}(\RRn)$.
\end{lemma}

\begin{proof} Fix $\delta>0$ and any $x \in \RRn$ such that $Mf(x)<\infty$. 
The following estimate is well-known:
\begin{equation}\label{94}
\int_{B(x,\delta)}\frac{|f(y)|}{|x-y|^{n-\alpha}}\, dy \leq \frac n\alpha \delta^\alpha Mf(x),
\end{equation}
see e.g. \cite[Inequality (3.1.1)]{adamshedberg}. Hence, 
by  inequalities \eqref{holder} and \eqref{49},   %Lemmas \ref{lemma2} and \ref{lemma6}, %and equation \eqref{92},
\begin{align}\label{101}
\int_{\RRn}\frac{|f(y)|}{|x-y|^{n-\alpha}}\,dy & = \int_{B(x,\delta)}\frac{|f(y)|}{|x-y|^{n-\alpha}}\, dy + \int_{\RRn \setminus B(x,\delta)}\frac{|f(y)|}{|x-y|^{n-\alpha}}\, dy 
\\ \nonumber & \leq \frac n\alpha \delta^\alpha Mf(x) + 2 \|f\|_{L^{\phi (\cdot)}(\RRn)} \bigg\|\frac{\chi_{\RRn \setminus B(x,\delta)}(\cdot)}{|x-\cdot|^{n-\alpha}}\bigg\|_{L^{\tp (\cdot)}(\RRn)}
\\ \nonumber & \leq \frac n\alpha \delta^\alpha Mf(x) + 2 \|f\|_{L^{\phi (\cdot)}(\RRn)}\lambda (x,\delta).
\end{align}
The choice 
$$\delta =  \phi\bigg(x, \frac{Mf(x)}{ \|f\|_{L^{\phi (\cdot)}(\RRn)}}\bigg)^{-\frac 1n}$$
in equation \eqref{101} yields:
\begin{equation}\label{102}
|I_\alpha f(x)| \leq \frac n\alpha \frac{Mf(x)}{ \phi\Big(x, \frac{Mf(x)}{ \|f\|_{L^{\phi (\cdot)}(\RRn)} }\Big)^{\frac \alpha n}} + 2 \|f\|_{L^{\phi (\cdot)}(\RRn)} \omega \bigg(x, \frac{Mf(x)}{ \|f\|_{L^{\phi (\cdot)}(\RRn)}}\bigg).
\end{equation}
Now, observe that
\begin{align}\label{103}
 \frac t{\phi(x,t)^{\frac \alpha n}} \leq c \,\omega (x, t) \quad \text{for $t>0$,}
%\frac 1t \frac 1{\phi(x,t^{-1})^{\frac \alpha n}} \leq c \omega (x, t^{-1}),
\end{align}
where $c$ is the constant from inequality \eqref{85}. Indeed, owing to inequality  \eqref{85},

%\todo[inline]{A: we need $\phi^{-1}(\phi (t)) \geq t$, i.e. $\phi^{-1}$ right-continuous}

\begin{equation*}
\phi^{-1}(x, s^{-1}) s^{\frac \alpha n}\leq c \,\lambda \big(x, s^{\frac 1n}\big) \quad \text{for $s>0$,}
%
%\, \Phi (x,s),
\end{equation*}
whence inequality \eqref{103} follows, on setting $s=\frac 1{\phi(x,t)}$.
\\ Exploiting inequality \eqref{103}, with
$$t= \frac{Mf(x)}{\|f\|_{L^{\phi (\cdot)}(\RRn)}},$$
to estimate the first addend on the right-hand side of inequality \eqref{102} tells us that
\begin{equation}\label{104}
|I_\alpha f(x)| \leq c  \|f\|_{L^{\phi (\cdot)}(\RRn)} \omega \bigg(x, \frac{Mf(x)}{ \|f\|_{L^{\phi (\cdot)}(\RRn)}}\bigg)+ 2 \|f\|_{L^{\phi (\cdot)}(\RRn)} \omega \bigg(x, \frac{Mf(x)}{ \|f\|_{L^{\phi (\cdot)}(\RRn)}}\bigg)
\end{equation}
for some constant $c=c(n, \alpha, \beta)$, namely \eqref{100}.
\end{proof}

\begin{lemma}\label{cistro}
There exist positive constants $c_1=c_1(n, \alpha, \sigma)$ and $c_2=c_2(n, \alpha, \sigma)$  such that
\begin{equation}\label{114}
c_1 k\, H_{\frac n\alpha}(x,t) \leq \omega ( x,  t) \leq c_2 k\, H_{\frac n\alpha}(x,t) \quad \quad \text{for $x \in \RRn$ and $t \geq 0$.}
\end{equation}
Here, $\sigma$ and $k$ denote the constants appearing in equation \eqref{30}.
\end{lemma}
\begin{proof}
Define the generalized Young function $\phi_k : \RRn \times [0, \infty) \to [0, \infty]$ as
%Define the function 
%\begin{equation}\label{110}
%H_\alpha(x,s) = \bigg(\int_0^s \bigg(\frac t{\phi (x,t)}\bigg)^{\frac \alpha {n-\alpha}}\, dt\bigg)^{\frac{n-\alpha}n}.
%\end{equation}
%Notice that, in particular,
%$$H_1 = H,$$
%where $H$ is the function given by \eqref{sobconj}.
%Also, set
\begin{equation}\label{111}
\phi_k(x,t) = \phi\Big(x, \frac tk\Big) \quad \text{for $x \in \RRn$ and $t \geq 0$.}
\end{equation}
Moreover, let $H_{\frac n\alpha ,k} : \RRn \times [0, \infty) \to [0, \infty)$ be the function defined as in \eqref{Ha}, with $\phi$ replaced by $\phi_k$. Namely,
\begin{equation}\label{112}
H_{\frac n\alpha ,k}(x,t) = \bigg(\int_0^t \bigg(\frac \tau{\phi_k (x, \tau)}\bigg)^{\frac \alpha {n-\alpha}}\, d \tau\bigg)^{\frac{n-\alpha}n} \quad \quad \text{for $x \in \RRn$ and $t \geq 0$.}
\end{equation}
Observe that
\begin{equation}\label{112'}
\widetilde {\phi_k} (x,t) = \tp (x, tk)\quad  \quad \text{for $x \in \RRn$ and $t \geq 0$,}
\end{equation}
where 
 $\widetilde {\phi_k}$ denotes the Young conjugate of the function $\phi_k$. Also,
\begin{equation}\label{113}
H_{\frac n\alpha, k}(x,t) = k H_{\frac n\alpha} \Big(x, \frac tk\Big) \quad \text{for $x \in \RRn$ and $t \geq 0$.}
\end{equation}
By \cite[Lemma 2]{CianchiStroff}, applied to the function $\phi_k(x, \cdot)$ for each fixed $x\in \RRn$, and equation \eqref{5}, there exist positive constants $c_1=c_1(n, \alpha, \sigma)$ and $c_2=c_2(n, \alpha, \sigma)$  such that
%\begin{equation}\label{114'}
%c_1 H_{\alpha, k}(x,s) \leq \Phi \Big(x, \frac 1{\phi_k (x, s)}\Big) \leq c_2 H_{\alpha, k}(x,s).
%\end{equation}
%Hence, by the very definition of the function $\omega$,
$$
c_1 H_{\frac n\alpha, k}(x,t) \leq \omega \Big( x, \frac tk\Big) \leq c_2 H_{\frac n\alpha, k}(x,t) \quad  \quad \text{for $x \in \RRn$ and $t \geq 0$.}
$$
Hence, equation \eqref{114} follows, via inequalities \eqref{113}.
\end{proof}

%
%\begin{lemma}\label{lemma8} There exists a constant $c>0$ such that 
%    \begin{align}\label{120}
%\int_{\{Mf>t\}}\phi (x,t)\,dx \leq \int_{\RRn}\phi (x, c|f(x)|)\,dx \quad \text{for $t \geq 0$,}
%    \end{align}
%if   
%\begin{equation}\label{120'}
%\|f\|_{L^{\phi (\cdot)}(\RRn)}\leq \frac 1c.
%\end{equation}
%\end{lemma}
%This is a consequence of Proposition \ref{pro:maxfunction}, part (b), applied with $f$ replaced by $\frac f{b}$.

\begin{lemma}\label{lemma9}
%Let $k$ be constant appearing in equation \eqref{30}.  Then, there exist constants $c$ and $c'$, independent of $k$, such that
%\begin{equation}\label{125}
%\int_{\{|I_\alpha f(x) |>c'k H_{\frac n\alpha}(x, t)\}}\phi (x,t)\, dx \leq \int _{\RRn} \phi (x, c|f(x)|)\, dx \quad \text{for $t \geq 0$,}
%\end{equation}
Let  $\sigma$ and $k$ be the constants from equation \eqref{30}, and let $b>0$. There exists a constant $c=c(n, \alpha, \beta, \sigma)$ such that, if
%for every  function $f\in L^{\phi (\cdot)}(\RRn)$ such that
\begin{equation}\label{120'}
\|f\|_{L^{\phi (\cdot)}(\RRn)}\leq b,
\end{equation}
then 
\begin{align}\label{127}
|I_\alpha f(x)|    \leq   ck\max\{b,1\} H_{\frac n\alpha}\big(x,  Mf(x)\big)   \quad \text{for a.e. $x\in \RRn$.}
\end{align}
\end{lemma}
\begin{proof}  
%By Proposition \ref{pro:maxfunction}, part (ii), applied with $f$ replaced by $\frac f{b}$,   there exists a positive constant $c$ such that
%    \begin{align}\label{120}
%\int_{\{Mf>t\}}\phi (x,t)\,dx \leq \int_{\RRn}\phi (x, c|f(x)|)\,dx \quad \text{for $t \geq 0$,}
%    \end{align}
%if $f$ fulfills inequality \eqref{120'}.
%Next,
Inequalities \eqref{100} and \eqref{114} ensure that there exists a constant $c=c(n, \alpha, \beta,  \sigma)$ such that
\begin{align}\label{126}
|I_\alpha f(x)| \leq   c k\, \|f\|_{L^{\phi (\cdot)}(\RRn)}  H_{\frac n\alpha}\bigg(x, \frac{Mf(x)}{ \|f\|_{L^{\phi (\cdot)}(\RRn)}}\bigg)\quad \text{for a.e. $x\in \RRn$. }
\end{align}
Since, for each fixed $x\in \RRn$, the function 
\begin{equation}\label{Hconc}
\text{$t \mapsto H_{\frac n\alpha}(x, t)$ is concave and $H_{\frac n\alpha}(x,0)=0$,}
\end{equation}
 the function 
\begin{equation}\label{monotone}
t\mapsto \frac 1t\,H_{\frac n\alpha} (x,t) \qquad \text{is non-increasing.}
\end{equation}
By  assumption 
\eqref{120'} and  properties \eqref{monotone} and \eqref{5}, inequality \eqref{126} entails that
\begin{align*}%\label{127}
|I_\alpha f(x)| \leq ckb H_{\frac n\alpha}\Big(x,  \frac{ Mf(x)}b\Big)  \leq  ckb\max\Big\{1, \frac 1b\Big\} H_{\frac n\alpha}\big(x,  Mf(x)\big)   \quad \text{for a.e. $x\in \RRn$.}
\end{align*}
Hence, inequality \eqref{127} follows.
%whence
%\begin{align}\label{128}
%\frac kc H^{-1}\Big(x, \frac{c'}k |I_\alpha f(x)|\Big) \leq Mf(x) \quad \text{for $x\in \RRn$,}
%\end{align}
%where we have set $c'=\frac c{c''}$.
%\\ From inequalities \eqref{120} and \eqref{128} one deduces that
% \begin{align*} 
%\int_{\{\frac kc H^{-1}\Big(x, \frac{c'}k |I_\alpha f(x)|\Big)>t\}}\phi_k(x,t)\,dx \leq \int_{\RRn}\phi_k(x, c|f(x))|\,dx \quad \text{for $t \geq 0$}
%    \end{align*}
%for every function $f$ fulfilling condition \eqref{120'}
%Moreover, 
%an application of Proposition \ref{pro:maxfunction}, part (ii), applied with $f$ replaced by $\frac f{b}$, tells us that, under assumption  \eqref{120'},
%\\ From inequalities \eqref{120} and \eqref{127} one deduces that
%\begin{align*} 
%\int_{\{|I_\alpha f(x)| >  c'k H_{\frac n\alpha}(x,  t)\} }\phi (x,t)\,dx \leq \int_{\RRn}\phi (x, c|f(x)|)\,dx \quad \text{for $t \geq 0$,}
%    \end{align*}
%namely inequality  \eqref{125}.
\end{proof}

%As a consequence of Proposition \ref{pro:maxfunction}, Part (a), one has the following lemma.
%
%\begin{lemma}\label{lemma10} There exists $\gamma >0$ such that
%\begin{align}\label{130}
%\phi (x, \gamma  M_Bf(x)) \leq M_B \big(\phi (\cdot ,  f(\cdot))\big)
%\end{align}
% for every   $x\in \RRn$ and $x\ni B$, provided that $\| f\|_{L^{\phi(\cdot)}(\RRn)}\leq 1$.
%\end{lemma}

\begin{lemma}\label{lemma11}
 Let  $z\in \RRn$ and $t\geq 0$. Assume that  $f \in L^{\phi(\cdot)}(\RRn)$.  Let $B_z$ be a ball such that $z\in B_z$ and
\begin{equation}\label{142}
H _{\frac n\alpha}\big(z, M_{B_z}f\big)>t.
\end{equation}
Assume that
\begin{equation}\label{oct61}
0 <\eta \leq  \min\{\tfrac \beta 4, \tfrac \gamma \beta\},
\end{equation}
where $\gamma$ is the constant from Lemma \ref{pro:maxfunction}. If 
%There exists a constant $\eta=\eta(n,   \beta)$ such that, if
 $x\in B_z$ and 
\begin{equation}\label{120''}
\| f\|_{L^{\phi(\cdot)}(\RRn)}\leq \eta,
\end{equation}
then
\begin{equation}\label{145}
\phi\big(x, H^{-1}_{\frac n\alpha}(x, \beta t)\big) \leq M_{B_z}\big(\phi (\cdot, \tfrac \beta \gamma f(\cdot)).
\end{equation}
\end{lemma}

\begin{proof}
Assume that the function $f$ fulfills condition \eqref{120''}, with $\eta$ to be chosen later. Fix $x \in B_z$.
If 
$$0 \leq s \leq \beta  \phi^{-1}(z, |B_z|^{-1}),$$
then, by inequality \eqref{feb2},
\begin{align}\label{146}
H_{\frac n\alpha}(x,s) & =   \bigg(\int_0^s \bigg(\frac t{\phi (x,t)}\bigg)^{\frac \alpha {n-\alpha}}\, dt\bigg)^{\frac{n-\alpha}n} \geq  \bigg(\int_0^s \bigg(\frac t{\phi (z,t/\beta)}\bigg)^{\frac \alpha {n-\alpha}}\, dt\bigg)^{\frac{n-\alpha}n}
\\ \nonumber & = \beta \bigg(\int_0^{\frac s\beta} \bigg(\frac t{\phi (z,t)}\bigg)^{\frac \alpha {n-\alpha}}\, dt\bigg)^{\frac{n-\alpha}n} = \beta H_{\frac n\alpha}\Big(z, \frac s\beta\Big).
\end{align}
Hence,
\begin{equation}\label{147}
H^{-1}_{\frac n\alpha}(x, \beta t) \leq \beta H^{-1}_{\frac n\alpha}(z,t) 
\end{equation}
if 
\begin{equation}\label{147'}
0 \leq t \leq H_{\frac n\alpha}\big(z,  \phi^{-1}(z, |B_z|^{-1})\big).
\end{equation}
Thanks to assumption \eqref{120''}, inequalities \eqref{holder}, \eqref{80'} and  \eqref{6} imply that
\begin{align}\label{150}
M_{B_z}f & = \frac 1{|B_z|}\int_{B_z}|f(y)|\, dy \leq  \frac 2{|B_z|}\|f\|_{L^{\phi(\cdot)}(\RRn)}\|\chi_{B_z}\|_{L^{\tp(\cdot)}(\RRn)} 
\\ \nonumber &\leq \frac{4  \eta }{\beta |B_z|\tp^{-1}(z, |B_z|^{-1})} \leq \frac{4\eta}\beta \phi^{-1}(z, |B_z|^{-1}).
\end{align}
By inequalities \eqref{142} and \eqref{150},
\begin{equation}\label{151}
 t \leq H_{\frac n\alpha}(z, M_{B_z}f) \leq H_{\frac n\alpha}\Big(z,  \frac{4\eta}\beta \phi^{-1}(z, |B_z|^{-1})\Big).
\end{equation}
Thus, condition \eqref{147'}, and hence \eqref{147}, hold provided that
\begin{equation}\label{152}
H_{\frac n\alpha}\Big(z,  \frac{4\eta}\beta \phi^{-1}(z, |B_z|^{-1})\Big) \leq H_{\frac n\alpha}\big(z, \phi^{-1}(z, |B_z|^{-1})\big).
\end{equation}
The latter inequality is fulfilled if
\begin{equation}\label{153}
\eta \leq  \frac \beta 4.
\end{equation}
Inequalities \eqref{147} and  \eqref{142}  tell us that
\begin{align}\label{154}
H^{-1}_{\frac n\alpha}(x, \beta t) \leq \beta H^{-1}_{\frac n\alpha}(z, t) \leq \beta M_{B_z}f.
\end{align}
Hence, owing to
 Proposition \ref{pro:maxfunction},  
\begin{align}\label{155}
\phi\big(x, H^{-1}_{\frac n\alpha}(x, \beta t) \big) \leq \phi\big(x, \beta M_{B_z}f\big) =  \phi\big(x, \gamma  M_{B_z}\big(\tfrac \beta \gamma f\big)\big)
  \leq M_{B_z}\big(\phi (\cdot, \tfrac \beta \gamma f(\cdot)\big),
\end{align}
if $\eta$ satisfies both inequality \eqref{153} and 
\begin{equation}\label{155'} 
\eta \leq   \frac \gamma \beta.
\end{equation}
Hence, inequality \eqref{145} follows. 
\end{proof}

\iffalse
\begin{theorem}(Weak type inequality for Riesz potentials)\label{lemma12}
There exists a constant $c=c(n, \alpha, \beta, k)$ such that
\begin{equation}\label{176'}
\int_{\{|I_\alpha f|>t\}}\phi(x, H^{-1}_\alpha(x, ct))\, dx \leq \int_{\RRn}\phi(x, |f(x)|)\, dx
\end{equation}
if $\|f\|_{L^{\phi(\cdot)}(\RRn)}\leq 1$.
\end{theorem}
\todo[inline]{A: check the sharpness of this inequality}
\fi

\begin{proof}[Proof of Theorem \ref{lemma12}] Thanks to Proposition \ref{replacephi}, it suffices to prove inequality  \eqref{main1} with $\phi$ replaved with the normalized function $\overline \phi$ given by \eqref{july7}. For ease of notation,  we denote $\overline \phi$ simply by $\phi$. 
\\
Fix $t \geq 0$. Let $\eta$ be as in \eqref{oct61}, let $k$ be the constant in definition \eqref{30} of the function $\psi$, and let $c$ be the constant from
inequality \eqref{127}.
Assume that 
\begin{equation}\label{oct60}
\| f\|_{L^{\phi(\cdot)}(\RRn)}\leq \eta.
\end{equation}
By Lemma \ref{lemma9}, applied with $b=\eta$,
\begin{equation}\label{141}
\{|I_\alpha (x)| > c't\} \subset \{H_{\frac n\alpha}(x, Mf(x))>t\},
\end{equation}
where $c'= ck\max\{\eta,1\}$.  
\\ For each $z\in \{H_{\frac n\alpha}(x, Mf(x))>t\}$, let $B_z$ be a ball such that $z\in B_z$ and inequality \eqref{142} holds. By Besicovitch's covering theorem, there exists a countable sub-covering $\{B_i\}$ of the covering $\{B_z\}$ enjoying the bounded overlap property.
Thus, by Lemma \ref{lemma11},
 inequality \eqref{145} and this  property of  the covering, there exists a constant $c''=c''(n)$ such that
\begin{align}\label{173}
\int_{\{|I_\alpha (x)| > c' t\}} \phi (x, H^{-1}_{\frac n\alpha}(x,  \beta  t))\, dx & \leq \int_{\{H_{\frac n\alpha}(x, Mf(x))>t\}} \phi (x, H^{-1}_{\frac n\alpha}(x,   \beta t))\, dx 
\\ \nonumber & \leq \sum_{i=1}^\infty \int_{B_i} \phi (x, H^{-1}_{\frac n\alpha}(x,   \beta  t))\, dx \leq 
\sum_{i=1}^\infty
\int_{B_i} \dashint _{B_i}\phi(y, \tfrac \beta{ \gamma }|f(y)|)\, dy \,dx 
\\   \nonumber & =
\sum_{i=1}^\infty
\int_{B_i} \phi(y, \tfrac \beta{ \gamma }|f(y)|)\, dy \leq c'' \int_{\RRn}\phi(y, \tfrac \beta{ \gamma }|f(y)|)\, dy
\\   \nonumber &  \leq \int_{\RRn}\phi(y, \tfrac \beta{ \gamma } \max\{c'',1\}|f(y)|)\, dy \leq  \int_{\RRn}\phi(y, c'''|f(y)|)\, dy,
\end{align}
provided that the constant $c'''$ fulfills the inequality
$$c'''\geq \frac {\beta \max\{c'',1\}}{ \gamma }.$$
%Set
%$$c''= \frac{\beta \max\{c,1\}}{ \gamma k}\min\Big{\frac 1c, \eta\\Big},$$
%where $c$ and $\eta$ are the constants in equations \eqref{120'} and \eqref{120''}. 
%Define 
%$$\mu = \min \{b, \eta\},$$
%where $b$ and $\eta$ are the constants in equations \eqref{120'} and \eqref{120''}. 
An application of inequality \eqref{173}, with $f$ replaced by $\frac f{c'''}$, yields
%
%
%Hence,  if
%$$c''\geq  \max \{\tfrac {\beta \max\{c,1\}}{ \gamma k}, \frac 1\mu\},$$
%then we have that
\begin{equation}\label{180}
\int_{\{|I_\alpha f (x)| >c' c'' t\}} \phi (x, H^{-1}_{\frac n\alpha}(x,  \beta  t))\, dx \leq   \int_{\RRn}\phi(y,  |f(y)|)\, dy,
\end{equation}
if
\begin{equation}\label{181}
\|f\|_{L^{\phi(\cdot)}(\RRn)} \leq \eta\, c'''.
\end{equation}
Now, choose  
$$c'''= \max \Big\{\frac {\beta \max\{c'',1\}}{ \gamma }, \frac 1\eta \Big\}.$$
Thereby, we deduce from inequality \eqref{180} that
\begin{equation}\label{182}
\int_{\{|I_\alpha f(x)| > c' c''' t\}} \phi (x, H^{-1}_{\frac n\alpha}(x,  \beta  t))\, dx \leq   \int_{\RRn}\phi(y,  |f(y)|)\, dy,
\end{equation}
if 
\begin{equation*}%\label{181}
\|f\|_{L^{\phi(\cdot)}(\RRn)} \leq 1.
\end{equation*}
Inequality \eqref{176'} is thus established.
\end{proof}

\begin{proof}[Proof of Proposition \ref{optweak}] As in the proof of Theorem \ref{lemma12}, we may assume that  $\varphi = \overline \varphi$. 
Consider the Young function $A : [0, \infty) \to [0, \infty]$ defined as
\begin{equation}\label{theta}
A(t) = \sup_{x\in \RRn}  \phi (x, t) \qquad \text{for $t\geq 0$.}
\end{equation}
Plainly,
\begin{equation}\label{theta2}
A (t) \geq    \phi (x, t) \qquad \text{for $x\in \RRn$ and $t\geq 0$,}
\end{equation}
and, since $\phi$ is normalized,
\begin{equation}\label{theta1}
A(t) =   \phi (x,t) = \phi_\infty (t) \qquad \text{if $0 \leq t \leq 1$,}
\end{equation}
for $x\in \RRn$.
As a consequence of the latter identity,  there exists $t_0>0$ such that
\begin{equation}\label{theta3}
\widetilde A (t) = \widetilde  \phi (t) \qquad \text{if  $0 \leq t \leq t_0$.}
\end{equation}
Choose trial functions $f$ in inequality \eqref{optweak1} of the form $f(x)=g(\omega_n |x|^n)$,
where $g: [0, \infty) \to [0, \infty)$ is a bounded measurable function with bounded support. 
Then
\begin{align}\label{optweak4}
I_\alpha f(x) \geq 2^{\alpha -n} \int_{\{|y|\geq |x|\}}\frac{f(y)}{|y|^{n-\alpha}}\, dy = \omega_n^{\frac{n-\alpha}n}2^{\alpha -n} \int_{\omega _n |x|^n}^\infty g(s) s^{-\frac{n-\alpha}{n}}\, ds \quad \text{for $x \in \RRn$.}
\end{align}
Define the function  $Tg: [0, \infty) \to [0, \infty]$ as
$$Tg (r) =  \int_{\omega _n r^n}^\infty g(s) s^{-\frac{n-\alpha}{n}}\, ds \quad \text{for $r\in[0, \infty)$.}$$
Thus, inequality \eqref{optweak1} tells us that
\begin{equation}\label{optweak5}
\int_{\{x: Tg(|x|)>t\}}\vartheta (x,ct)\, dx \leq \int_{\RRn}\phi(x, g(\omega_n |x|^n))\, dx
\end{equation}
provided that 
\begin{equation}\label{optweak15} 
\|g(\omega_n |\cdot |^n)\|_{L^{\phi(\cdot)}(\RRn)}\leq 1.
\end{equation}
 By the monotone convergence theorem for Lebesgue integrals and the Fatou property of Musielak-Orlicz norms, replacing, for $t>0$, any measurable function $g: [0, \infty) \to [0, \infty)$ with the function $g_t :  [0, \infty) \to [0, \infty)$ given by 
\begin{align}\label{gt}
g_t(s)= \min\{t, g(s)\}\chi_{(0,t)}(s) \qquad \text{ for $s\geq 0$,}
\end{align}
 and passing to the limit as $t\to\infty$, tell us that inequality \eqref{optweak5} continues to hold for every nonnegative measurable function $g$ fulfilling inequality \eqref{optweak15}.
\\ Assume that $Tg(0)=\infty$, the case when $Tg(0)<\infty$ requiring just minor variants.
 Let $\{r_j\}$ be a non-increasing sequence in $[0, \infty)$ such that
$$Tg(r_j)= 2^j \qquad \text{for $j \in \mathbb Z$.}$$
%with the convention that, if $Tg(0)<\infty$, then  $r_j=0$ for all $j$ such that $2^j>Tg(0)$. 
Define the sequence of functions $\{g_j\}$ by
\begin{align}\label{optweak7}
g_j = g \chi_{(\omega_n r_j^n, \omega_n r_{j-1}^n)}
\end{align}
for $j \in \mathbb Z$. Inasmuch as
\begin{align}\label{optweak8}
Tg(r) \leq Tg(r_j) \leq 2^j \qquad \text{for $r \in (r_j, r_{j-1})$,}
\end{align}
one has that
\begin{align}\label{optweak9}
\int_{\RRn}\vartheta \Big(x,\frac{cTg(|x|)}4\Big)\, dx  = \sum_{j\in \mathbb Z}\int_{\{x: r_j < |x|<r_{j-1}\}}\vartheta \Big(x,\frac{cTg(|x|)}4\Big)\, dx& \leq  \sum_{j\in \mathbb Z}\int_{\{x: r_j < |x|<r_{j-1}\}}\vartheta \Big(x,\frac{c2^j}4\Big)\, dx 
\\ \nonumber & = \sum_{j\in \mathbb Z}\int_{\{x: r_j < |x|<r_{j-1}\}}\vartheta (x,c 2^{j-2})\, dx. 
\end{align}
Assume that $r_j <|x|<r_{j-1}$ for some $j \in \mathbb Z$. Thus,
\begin{align}\label{optweak10}
Tg_{j-1}(|x|)&\geq \int_{\omega _n r_{j-1}^n}^\infty g_{j-1}(s) s^{-\frac{n-\alpha}{n}}\, ds= 
 \int_{\omega _n r_{j-1}^n}^\infty g(s) s^{-\frac{n-\alpha}{n}}\chi_{(\omega_n r_{j-1}^n, \omega_n r_{j-2}^n)}(s)\, ds
\\ \nonumber & =  \int_{\omega _n r_{j-1}^n}^{\omega _n r_{j-2}^n} g(s) s^{-\frac{n-\alpha}{n}}\, ds = 
Tg(r_{j-1})- Tg(r_{j-2})= 2^{j-2}.
\end{align}
Hence,
\begin{align}\label{optweak11}
\{x: r_j <|x|<r_{j-1}\} \subset \{x: Tg_{j-1}(|x|)>2^{j-2}\}.
\end{align}
Coupling the latter equation with inequality  \eqref{optweak5}, applied with $g$ replaced by $g_{j-1}$, tells us that
\begin{align}\label{optweak12}
\int_{\{x: r_j <|x|<r_{j-1}\} }\vartheta (x,c2^{j-2})\, dx \leq  \int_{\{x: Tg_{j-1}(|x|)>2^{j-2}\}}\vartheta (x,c2^{j-2})\, dx\leq  \int_{\RRn}\phi(x, g_{j-1}(\omega_n |x|^n))\, dx.
\end{align}
Observe that the application of inequality  \eqref{optweak5} is legitimate, since $g_{j-1} \leq g$, whence $\|g_{j-1}(\omega_n |\cdot|^n)\|_{L^{\phi(\cdot)}(\RRn)}\leq 1$ if $g$ satisfies condition \eqref{optweak15}. From inequalities \eqref{optweak9} and \eqref{optweak12} we deduce that 
\begin{align}\label{optweak16}
\int_{\RRn}\vartheta \Big(x,\frac{cTg(|x|)}4\Big)\, dx  \leq  \sum_{j\in \mathbb Z} \int_{\RRn}\phi(x, g_{j-1}(\omega_n |x|^n))\, dx &=  \sum_{j\in \mathbb Z}  \int_{\{x: r_j < |x|<r_{j-1}\}}\phi(x, g(\omega_n |x|^n))\, dx \\ \nonumber &= \int_{\RRn}\phi(x, g(\omega_n |x|^n))\, dx
\end{align}
for every $g$ fulfilling condition \eqref{optweak15}. Applying the latter inequality with $g(\omega_n |x|^n)$ replaced by $\frac{g(\omega_n |x|^n)}{\|g(\omega_n |\cdot|^n)\|_{L^{\phi(\cdot)}(\RRn)}}$ tells us that
\begin{align}\label{optweak17}
\bigg\| \int_{\omega _n |x|^n}^\infty g(s) s^{-\frac{n-\alpha}{n}}\, ds\bigg\|_{L^{\vartheta (\cdot)}(\RRn)} \leq\frac 4c \|g(\omega_n |x|^n)\|_{L^{\phi(\cdot)}(\RRn)}
\end{align}
for every $g$ such that $g(\omega_n |x|^n) \in L^{\phi(\cdot)}(\RRn)$. Owing to inequality \eqref{theta2},
\begin{align}\label{optweak18}
 \|g(\omega_n |x|^n)\|_{L^{\phix}(\RRn)} \leq  \|g(\omega_n |x|^n)\|_{L^{A}(\RRn)} = \|g\|_{L^{A}(0,\infty)}.
\end{align}
On the other hand, if  $g$ vanishes in $(0, \omega_n)$, then
\begin{align}\label{optweak19}
\bigg\| \int_{\omega _n |x|^n}^\infty g(s) s^{-\frac{n-\alpha}{n}}\, ds\bigg\|_{L^{\vartheta (\cdot)}(\RRn)} & \geq 
\bigg\| \chi_{B(0,1)}(x)\int_{\omega_n |x|^n}^\infty g(s)s^{-\frac{n-\alpha}{n}}\, ds  \bigg\|_{L^{\vartheta (\cdot)}(\RRn)}\\ \nonumber & = \| \chi_{B(0,1)}\|_{L^{\vartheta (\cdot)}(\RRn)}\int_{\omega_n}^\infty g(s) s^{-\frac{n-\alpha}{n}}\, ds.
\end{align}
Combining  inequalities \eqref{optweak17} -- \eqref{optweak19} implies that
\begin{align}\label{optweak20}
\sup _g \frac{ \int_{\omega_n}^\infty g(s) s^{-\frac{n-\alpha}{n}}\, ds }{\|g (s)\|_{L^{A}(\omega_n, \infty)}} \leq \frac{c}{\| \chi_{B(0,1)}\|_{L^{\vartheta (\cdot)}(\RRn)}}
\end{align}
for some constant $c$ and for every  nonnegative function $g\in L^{A}(\omega_n, \infty)$.
Thanks to  inequalities \eqref{optweak20} and  \eqref{revholder},  applied wtih $\phi$ replaced with $A$,
\begin{align}\label{optweak21}
\big\| s^{-\frac{n-\alpha}{n}}\big\|_{L^{\widetilde A}(\omega_n, \infty)} < \infty.
\end{align}
By a chain analogous to (and, in fact, simpler than) \eqref{51},  inequality \eqref{optweak21} is equivalent to
\begin{equation}\label{optweak22}
\int_0\frac{\widetilde A(t)}{t^{1+\frac n{n-\alpha}}}\, dt <\infty,
\end{equation}
which, by \cite[Lemma 2.3]{cianchi_ibero},  is in turn equivalent to
\begin{equation}\label{optweak23}
\int_0\bigg(\frac{t}{A(t)}\bigg)^{\frac \alpha{n-\alpha}}\, dt <\infty.
\end{equation}
Inequality \eqref{conv0a} hence follows, owing to equation  \eqref{theta1}.

\end{proof}

\section{Proof of the Sobolev inequalities}\label{sec:proof}

The link between the weak-type inequality for Riesz potentials from the previous section and the Sobolev inequalities of Theorems \ref{thm:main} and \ref{thm:main'} is provided by the representation formula which is the subject of the next lemma. The formula is standard for compactly supported functions. The punctum of this lemma is the broader class of admissible functions. 

\begin{lemma}\label{lemma30} Assume that $u \in W^{1,1}_{\rm loc}(\RRn)$ is such that
\begin{equation} \label{june3}
 |\{|u|>t\}|<\infty \qquad \text{for $t>0$,}
\end{equation}
and 
\begin{equation}\label{july100}
\int_{\RRn \setminus B(0,1)}\frac{|\nabla u(x)|}{|x|^{n-1}}\, dx <\infty.
\end{equation}
%Let $\phi$ and $u$ be as in Theorem  \ref{thm:main}. 
Then,
\begin{align}\label{170bis}
u(x) = \frac 1{n \omega_n } \int_{\RRn } \frac{\nabla u(y)\cdot (x-y)}{|x-y|^n}\, dy
%|u(x)| < \kappa I_1(|\nabla u|) (x) \
\quad \text{for a.e. $x\in \RRn$.} 
\end{align}
Here, the dot $\lq\lq \cdot"$ stands for scalar product in $\RRn$.
\end{lemma}
\begin{proof} To begin with,  
 recall that, since $u \in W^{1,1}_{\rm loc}(\RRn)$,  for a.e. $x \in \RRn$ the function 
\begin{equation}\label{june17}
[0, \infty) \ni t \mapsto u (x+ \nu r)
\end{equation}
is   absolutely continuous, for $\mathcal H^{n-1}$-a.e. $\nu \in \mathbb S^{n-1}$, on every interval $[0, \ell]$ with $\ell >0$. Furthermore,   
\begin{equation}\label{june13}
\frac{d}{dr} u (x+ \nu r) = \nabla u (x+ \nu r) \cdot \nu \qquad \text{for a.e. $r>0$,}
\end{equation}
see e.g, \cite[??]{Ziemer}. 
 Fix any $x \in \RRn$ for which this property holds. 
%Inasmuch as 
%\begin{equation} \label{june3bis}
% |\{u>t\}|<\infty \qquad \text{for $t>0$,}
%\end{equation}
We claim that 
\begin{equation} \label{june4}
\liminf _{r \to \infty} |u(x+\nu r)|= 0 \quad \text{for $\mathcal H^{n-1}$-a.e. $\nu \in \mathbb S^{n-1}$.}
\end{equation}
To verify this claim, set
\begin{equation} \label{june5}
E= \{\nu \in \mathbb S^{n-1}: \liminf _{r \to \infty} |u(x+\nu r)|> 0\},
\end{equation}
and 
$$E_j=  \{\nu \in \mathbb S^{n-1}: \liminf _{r \to \infty} |u(x+\nu r)|> 1/j\}$$
for $j \in \mathbb N$.
Assume, by contradiction, that
$
\mathcal H^{n-1}(E)>0
$.
Since
$
E= \bigcup_{j\in \mathbb N} E_j
$,
one hence infers that
\begin{equation}\label{june8}
0<\mathcal H^{n-1}(E) = \lim_{j\to \infty} \mathcal H^{n-1}(E_j).
\end{equation}
Therefore, there exists $j_0 \in \mathbb N$  such that
 \begin{equation}\label{june9}
 \mathcal H^{n-1}\big( E_{j_0}\big)>0.
\end{equation}
In particular, this implies that
\begin{equation}\label{june10}
\int _0^\infty \chi_{ \{|u(x+\nu r)|> 1/{j_0}\}}(\nu, r) r^{n-1}\, dr = \infty
\end{equation}
for every $\nu \in E_{j_0}$. Thus,
\begin{align}\label{june11}
 |\{|u|>1/{j_0}\}| & = \int_{\{|u|>1/{j_0}\}}dy = \int_{\mathbb S^{n-1}} \int_0^\infty \chi_{ \{|u(x+\nu r)|> 1/{j_0}\}}(\nu, r) r^{n-1}\, dr\, d\mathcal H^{n-1}(\nu) \\ \nonumber & \geq  \int_{ E_{j_0}} \int_0^\infty \chi_{ \{|u(x+\nu r)|> 1/{j_0}\}}(\nu, r) r^{n-1}\, dr\, d\mathcal H^{n-1}(\nu)=\infty,
\end{align}
and this contradicts assumption \eqref{june3}. Property \eqref{june4} is therefore established.
\\ Next, observe that%, for a.e. $x \in \RRn$,
\begin{equation}\label{june12}
\int_0^\infty |\nabla u (x+ \nu r) \cdot \nu|\, dr <\infty \qquad   \text{for $\mathcal H^{n-1}$-a.e. $\nu \in \mathbb S^{n-1}$.}
\end{equation}
%To verify this assertion, recall that, since $u \in W^{1,1}_{\rm loc}(\RRn)$, one has that, for a.e. $x \in \RRn$, the function 
%\begin{equation}\label{june17}
%[0, \infty) \ni t \mapsto u (x+ \nu r)
%\end{equation}
%is   absolutely continuous on every interval $[0, \ell]$ with $\ell >0$,  for $\mathcal H^{n-1}$-a.e. $\nu \in \mathbb S^{n-1}$, and 
%\begin{equation}\label{june13}
%\frac{d}{dr} u (x+ \nu r) = \nabla u (x+ \nu r) \cdot \nu \qquad \text{for a.e. $r>0$,}
%\end{equation}
%see e.g, \cite[??]{Ziemer}. 
In order to prove inequality \eqref{june12}, notice the following chain:
\begin{align}\label{june14}
\int_{\mathbb S^{n-1}}\int_1^\infty  |\nabla u (x+ \nu r) \cdot \nu|\, dr \,  d\mathcal H^{n-1}(\nu)& = 
\int_{\mathbb S^{n-1}}\int_1^\infty \frac{ |\nabla u (x+ \nu r) \cdot \nu|}{r^{n-1}}\, r^{n-1} \,dr \,  d\mathcal H^{n-1}(\nu)
\\ \nonumber & = \int_{\RRn \setminus B_1(x)} \frac{|\nabla u(y)\cdot (y-x)|}{|y-x|^n}\, dy \leq \int_{\RRn \setminus B_1(x)} \frac{|\nabla u(y)|}{|y-x|^{n-1}}\, dy.
\end{align}
The last integral converges, thanks to assumption \eqref{july100}. Thereby, 
\begin{equation}\label{june15}
\int_1^\infty  |\nabla u (x+ \nu r) \cdot \nu|\, dr < \infty \quad \text{for $\mathcal H^{n-1}$-a.e. $\nu \in \mathbb S^{n-1}$.}
\end{equation}
On the other hand,
\begin{equation}\label{june16}
\int_0^1  |\nabla u (x+ \nu r) \cdot \nu|\, dr < \infty \quad \text{for $\mathcal H^{n-1}$-a.e. $\nu \in \mathbb S^{n-1}$,}
\end{equation}
since the function in \eqref{june17} is   absolutely continuous on $[0, 1]$ and property \eqref{june13} holds. Coupling inequality \eqref{june15} with \eqref{june16} yields \eqref{june12}.
%Altogether,
%\begin{equation}\label{june18}
%\int_0^\infty  |\nabla u (x+ \nu r) \cdot \nu|\, dr < \infty \quad \text{for $\mathcal H^{n-1}$-a.e. $\nu \in \mathbb S^{n-1}$.}
%\end{equation}
\\
Now, fix any $\nu \in \mathbb S^{n-1}$ such that equations \eqref{june4} and  \eqref{june12} hold. Therefore, there exists a sequence $\{r_j\}$ such that $\lim _{j \to \infty} r_j =\infty$ and $\lim_{j\to \infty}  u(x+\nu r_j)= 0$.
Inasmuch as 
\begin{equation}\label{june20}
u(x+\nu r_j) - u(x) = \int_0^{r_j} \nabla u (x+ \nu r) \cdot \nu\, dr,
\end{equation}
passing to the limit as $j \to \infty$ in the latter equation yields
\begin{equation}\label{june21}
u(x) = - \int_0^{\infty} \nabla u (x+ \nu r) \cdot \nu\, dr.
\end{equation}
Observe that the passage to the limit in the integral is legitimate, thanks to the dominated convergence theorem and to  equation \eqref{june12}. Integrating identity \eqref{june21}, which holds for $\mathcal H^{n-1}$-a.e. $\nu \in \mathbb S^{n-1}$, with respect to $\nu$ over  $ \mathbb S^{n-1}$ yields
\begin{equation}\label{june22}
n \omega_n u(x) = - \int_{\mathbb S^{n-1}} \int_0^{\infty} \nabla u (x+ \nu r) \cdot \nu\, dr\, d\mathcal H^{n-1}(\nu) = - \int_{\RRn } \frac{\nabla u(y)\cdot (y-x)}{|y-x|^n}\, dy.
\end{equation}
Equation \eqref{170bis} is thus established.
\end{proof}

We are now in a position to prove our Sobolev inequalities. We begin with the inequality in $\RRn$.

\begin{proof}[Proof of Theorem \ref{thm:main}]  Proposition \ref{replacephi} enables us to assume, without loss of generality,  that $\phi= \overline \phi$.
\\
Our assumptions on $\phi$ and $u$ ensure that the latter fulfills the hypotheses of Lemma \ref{lemma30}. In particular, condition \eqref{july100} is satisfied owing to the H\"older inequality \eqref{holder},  inequality \eqref{49} and the fact that $|\nabla u| \in L^{\phix}(\RRn)$. 
From equation
\eqref{170bis} one thus infers that
\begin{align}\label{170}
|u(x)| < \kappa I_1(|\nabla u|) (x) \quad \text{for a.e. $x\in \RRn$,}
\end{align}
for some constant $\kappa = \kappa (n)$.
Thereby, 
if, in addition,
\begin{equation}\label{170'}
\|\nabla u\|_{L^{\phi(\cdot)}(\RRn)} \leq 1,
\end{equation}
then Lemma \ref{lemma12}  yields:
\begin{align}\label{176''}
\int_{\{|u|\geq t\}}\phi_n(x, c't)\, dx = \int_{\{|u|\geq t\}} \phi(x, H^{-1}_n(x, c't))\, dx & \leq  \int_{\{ I_1(|\nabla u|)\geq t/\kappa \}} \phi(x, H^{-1}_n(x, c't))\, dx\\ \nonumber & \leq \int_{\RRn}\phi(x, |\nabla u|)\, dx,
\end{align}
where $c'= \frac c\kappa$, and $c$ is the constant from inequality \eqref{176'}.
Define the function $u_j : \RRn \to  [0, \infty)$ as
$$u_j = \max\{\min \{|u|-2^j, 2^j\},0\},$$
and the set
$$U_j =\{x\in \RRn: 2^j<|u(x)|\leq 2^{j+1}\}$$
for $j \in \mathbb Z$. Observe that
$$U _j \subset  \{x\in \RRn: u_{j-1}(x) = 2^{j-1}\}$$
for $j \in \mathbb Z$.  An application of inequality \eqref{176''}, with $u$ replaced with $u_{j-1}$, enables one to deduce that
\begin{align}\label{202}
\int_{\RRn} \phi_n(x, \tfrac {c'}4 |u|)\, dx & = \sum_{j\in \mathbb Z} \int_{U_j} \phi_n(x,  \tfrac {c'}4|u|)\, dx \leq  \sum_{j\in \mathbb Z}\int_{U_j} \phi_n(x, c'2^{j-1})\, dx 
\\ \nonumber &   \leq\sum_{j\in \mathbb Z} \int_{\{ u_{j-1}  \geq 2^{j-1}\}} \phi_n(x, c' 2^{j-1})\, dx \\ \nonumber &  \leq \sum_{j\in \mathbb Z}\int_{\RRn}\phi(x, |\nabla u_j|)\, dx   = \int_{\RRn}\phi(x, |\nabla u|)\, dx \leq 1,
\end{align}
where the last inequality holds thanks to assumption \eqref{170'}.
Inequality \eqref{main1} follows via  inequality \eqref{202}, applied with $u$ replaced with $\frac u{\|\nabla u\|_{L^{\phi(\cdot)}(\RRn)}}$.
\end{proof}

\begin{proof}[Proof of Proposition \ref{nec}] We may asssume, as in the proof of Theorem \ref{thm:main}, that $\phi=\overline \phi$. Let $A$ be the Young function defined as in \eqref{theta}.
Consider trial functions $u$ in inequality \eqref{nec1} of the form
$$u(x) = \int_{\omega_n |x|^n}^\infty g(s)s^{-\frac 1{n'}}\, ds \quad\quad \text{for $x\in \RRn$,}$$
where  $g: [0, \infty) \to [0, \infty)$ is a measurable bounded function, with bounded support contained in $[\omega_n, \infty)$. Thus $u$ is weakly differentiable, and 
$$|\nabla u(x)| = c g(\omega_n |x|^n) \quad \text{for a.e. $x\in \RRn$,}$$
for some constant $c=c(n)$.
\\
Inequalities \eqref{nec1}  and \eqref{theta2} yield:
\begin{equation}\label{351}
\bigg\| \int_{\omega_n |x|^n}^\infty g(s)s^{-\frac 1{n'}}\, ds  \bigg\|_{Y(\RRn)}\leq c \|g(\omega_n |x|^n)\|_{L^{A}(\RRn)}= c \|g (s)\|_{L^{A}(\omega_n, \infty)}
\end{equation}
for some constant $c$. Thanks to property \eqref{lattice},
\begin{align}\label{352}
\bigg\| \int_{\omega_n |x|^n}^\infty g(s)s^{-\frac 1{n'}}\, ds  \bigg\|_{Y(\RRn)}& \geq\bigg\| \chi_{B(0,1)}(x)\int_{\omega_n |x|^n}^\infty g(s)s^{-\frac 1{n'}}\, ds  \bigg\|_{Y(\RRn)}\\ \nonumber & = \| \chi_{B(0,1)}\|_{Y(\RRn)}\int_{\omega_n}^\infty g(s) s^{-\frac 1{n'}}\, ds.
%\\ \nonumber &= c\| \chi_{B(0,1)}\|_{Y(\RRn)}\int_{\RRn \setminus B(0,1)} g(\omega_n |y-x|^n)\, dy
\end{align}
Coupling inequality \eqref{351} with \eqref{352} implies that
\begin{align}\label{352'}
\sup _g \frac{ \int_{\omega_n}^\infty g(s) s^{-\frac 1{n'}}\, ds }{\|g (s)\|_{L^{A}(\omega_n, \infty)}} \leq \frac{c}{\| \chi_{B(0,1)}\|_{Y(\RRn)}}.
\end{align}
By the Fatou property of Luxemburg norms, replacing, for $t>0$, any measurable function $g: [0, \infty) \to [0, \infty)$ with the function $g_t$ defined as in \eqref{gt}, and passing to the limit as $t\to\infty$, tell us that inequality \eqref{352'} continues to hold for every $g \in L^{A}(\omega_n, \infty)$. The left-hand side of inequality \eqref{352'} agrees with the left-hand side of inequality \eqref{optweak20}, with $\alpha =1$. Inequality \eqref{conv0} hence follows, since, as shown in the proof of Theorem \ref{lemma12}, inequality \eqref{optweak20} implies \eqref{conv0a} for every $\alpha \in (0, n)$.
\end{proof}

We conclude with our Sobolev inequalities on domains. Apart from the use of the simplified function $\phi_{n,\diamond}$ instead of  $\phi_{n}$, Theorem \ref{thm:mainzero} is just a consequence of Theorem \ref{thm:main}.  The justification of the replacement of  $\phi_{n,\diamond}$ by $\phi_{n}$
is the same as at the beginning of the proof of Theorem \ref{thm:main'} below.

\begin{proof}[Proof of Theorem \ref{thm:main'}] Owing to equation \eqref{july19} of Proposition \ref{replacephiomega}, 
 we may assume that $\phi=\widehat \phi$, the function associated with $\phi$ as in formula \eqref{july11}.
By the same proposition, the function $\phi$, which we are assuming to agree with $\widehat \phi$, is a normalized generalized Young function satisfying assumption \eqref{conv0}. Hence, by Theorem \ref{lemma12}  and Remark \ref{rem-oct5}, 
there exists a positive constant  $c=c(n, \alpha, \beta)$  such that
\begin{equation}\label{176'bis}
\int_{\{|I_1 f|>t\}}\phi_{n, \diamond}(x,ct)\, dx \leq \int_{\RRn}  \phi(x, |f(x)|)\, dx \quad \text{for $t \geq 0$,}
\end{equation}
if $\|f\|_{L^{\phi(\cdot)}(\RRn)}\leq 1$.
\\ {\color{black} As a preliminary, observe that, owing to property \eqref{nov151}, 
\begin{align}
    \label{nov158}
    V^{1,\phix}(\Omega) \embedding V^{1,1}(\Omega).
\end{align}
 Moreover, since $\Omega$ is a John domain, one classically has that
\begin{align}
    \label{nov172}
\|u\|_{L^1(\Omega)} \leq c \big(\|u\|_{L^1(G)} + \|\nabla u\|_{L^1(\Omega)}\big)\end{align}
for any open set $G$ such that $\overline G \subset \Omega$ and some constant $c=c(\Omega, G)$.
Hence,  
\begin{align}\label{nov171}
    V^{1,\phix}(\Omega) \embedding L^1(\Omega),
\end{align}
and
the mean value $u_\Omega$ is well defined for every $u \in V^{1,\phix}(\Omega)$.}
\\ In order to prove inequality \eqref{main1'}, it suffices to show that there exist positive constants $c_1=c_1(\beta, \Omega)$ and $c_2=c_(\beta, \Omega)$,  such that
\begin{equation}\label{303}
\int_\Omega \phi_{n, \diamond}(x, c_1|u-u_\Omega|)\, dx \leq c_2,
\end{equation}
if 
\begin{equation}\label{304}
\|\nabla u\|_{L^{\phi (\cdot)}(\Omega)} \leq \frac 1{1+2\|1\|_{L^{\widetilde \phi (\cdot)}(\Omega)}}.
\end{equation}
Let $u_j : \Omega \to [0, \infty)$ be the function defined by
$$u_j = \max\{\min \{|u-u_\Omega|-2^j, 2^j\},0\},$$
and let
$$\Omega_j =\{x\in \Omega: 2^j<|u(x)-u_\Omega|\leq 2^{j+1}\}$$
for $j \in \mathbb Z$.
We have  that
\begin{align}\label{305}
\Omega_j \subset  \{x\in \Omega: u_{j-1}(x)=2^{j-1}\}
\end{align}
for $j \in \mathbb Z$. 
 By  \cite[Lemma 8.2.1]{DieHarHasRuz17}, there exist a ball  $B\subset \Omega$ and a constant $\kappa = \kappa (n)$ such that
\begin{align}\label{306}
|u_j(x) - (u_j)_B| < \kappa  I_1(|\nabla u_j|) (x) \quad \text{for a.e. $x\in \Omega$,}
\end{align}
for $j \in \mathbb Z$. 
Therefore, 
\begin{align}\label{307}
u_j(x)&  \leq |u_j(x) - (u_j)_B|  + (u_j)_B< \kappa  I_1(|\nabla u_j|) (x) +  \dashint_{B}|u-u_\Omega|\, dy
\\ \nonumber &  \leq \kappa  I_1(|\nabla u_j|) (x) + c' \int_\Omega |u-u_\Omega|\, dy \leq  \kappa  I_1(|\nabla u_j|) (x) +c'' \int_\Omega|\nabla  u|\, dy 
\\ \nonumber &  \leq \kappa  I_1(|\nabla u_j|) (x) + 2c'' \|1\|_{L^{\widetilde \phi (\cdot)}(\Omega)}\|\nabla u\|_{L^{\phi (\cdot)}(\Omega)} \leq c'''\big( I_1(|\nabla u_j|) (x) +1\big) \quad \text{for a.e. $x \in \Omega$,}
\end{align}
for some constants $c', c'', c'''$ depending on $\Omega$, and for $j \in \mathbb Z$. Notice that here we made use of the Poincar\'e inequality in $W^{1,1}(\Omega)$.
\\ Denote by $j_0$ the largest index $j$ such that $2^{j-2} \leq c'''$. Let $c$ be the constant from inequality \eqref{176'bis}, and set 
\begin{equation}\label{c3}
\overline c= \frac 1{c'''} \max\{1,c\}.
\end{equation}
%
% Denote by $j_2$ the largest $j$ such that 
%$2^{j-2} \leq \widehat c$. 
%\\ Let $c$ be the constant from inequality \eqref{176'} with $\alpha =1$ and let $j_1$ the largest index $j$ such that $c\,2^{j-2} \leq 1$. Let 
%$c_1 $ be such that $c_1  \,2^{j-2} \leq 1$ if $2^{j-2} \leq \widehat c$. Let $c_2 = \min\{c, c_1\}$.
Owing to equation \eqref{305} and inequality \eqref{176'bis} applied with $f= |\nabla u|$, the following chain holds:
\begin{align}\label{308}
\int_\Omega \phi_{n, \diamond}(x, \tfrac {\overline c} 8 |u-u_\Omega|)\, dx & = \sum_{j\in \mathbb Z}\int_{\Omega_j} \phi_{n, \diamond}(x, \tfrac {\overline c}8 |u-u_\Omega|)\, dx \leq  \sum_{j\in \mathbb Z}\int_{\Omega_j}\phi_{n, \diamond}(x, \overline c 2^{j-2})\, dx
\\ \nonumber & \leq \sum_{j\in \mathbb Z}\int_{\{u_{j-1} \geq 2^{j-1}\}} \phi_{n, \diamond}(x, \overline c 2^{j-2})\, dx \leq 
\sum_{j\in \mathbb Z}\int_{\{c'''(I_1(|\nabla u_{j-1}|) +1) \geq 2^{j-1}\}}\phi_{n, \diamond}(x, \overline c 2^{j-2})\, dx 
\\ \nonumber & \leq 
\sum_{j\in \mathbb Z}\int_{\{c'''  I_1(|\nabla u_{j-1}|) \geq 2^{j-2}\}} \phi_{n, \diamond}(x, \overline c 2^{j-2})\, dx 
+ \sum_{j \leq j_0}\int_\Omega \phi_{n, \diamond}(x, \overline c 2^{j-2})\, dx
\\ \nonumber & \leq 
\sum_{j\in \mathbb Z}\int_{\Omega} \phi (x,  |\nabla u_{j-1}|)\, dx 
+  \sum_{ j \leq j_0}\int_{\Omega}\phi_{n, \diamond}(x, \overline c 2^{j-2})\, dx.
\end{align}
Observe that the choice \eqref{c3} plays a role in the application of inequality \eqref{176'bis}.
Plainly,
\begin{equation}\label{310}
\sum_{j\in \mathbb Z}\int_{\Omega} \phi (x,  |\nabla u_{j-1}|)\, dx = \int_\Omega  \phi (x,  |\nabla u|)\, dx \leq 1
\end{equation}
%On the other hand,
%since $\overline \phi \in \Phi_1(\RRn)$, we have that $\overline \phi (x, 1) =1$, whence
On the other hand, since  $\phi(x,\cdot)$ is a Young function such that $\phi (x, 1) =1$,
$$\phi (x, t) \leq t \quad \text{if $0 \leq t \leq 1$,}$$
for  $x\in \RRn$.
Thus, $H_{n, \diamond}^{-1}(x,t) \leq t^{n'} \leq 1$ if $0 \leq t \leq 1$. As a consequence,
$$\phi_{n, \diamond} (x,t) =  \phi (x, H_{n, \diamond}^{-1}(x,t)) \leq \phi(x, t^{n'}) \leq t^{n'} \qquad \text{if $0 \leq t \leq 1$, }$$
for  $x\in \RRn$.
Hence, owing to equation \eqref{c3}
\begin{align}\label{309}
 \sum_{  c'''  \,  \geq 2^{j-2}}\int_{\Omega} \phi_{n, \diamond}(x, \overline c 2^{j-2})\, dx
\leq  \overline c^{n'} |\Omega| \sum_{j\leq j_0} 2^{n'(j-2)} <\infty.
\end{align}
Combining inequalities \eqref{308} -- \eqref{309} yields \eqref{303}.
\\{\color{black} Owing to embedding \eqref{nov158}, the embedding 
$$V^{1,\phix}(\Omega) \embedding L^{\phi_{n,\diamond}}(\Omega)$$
is a consequence of inequality \eqref{main1'}
and of  embedding \eqref{nov171}.
}
\end{proof}

\begin{proof}[Proof of Theorem~\ref{thm:maincompact}]
    It suffices to show that, if  $\{u_k\}$ is a sequence in $V^{1,\phi(\cdot)}(\Omega)$ such that $\norm{u_k}_{V^{1,\phi(\cdot)}(\Omega)} \le 1$, then it admits a Cauchy subsequence  in $L^{\vartheta(\cdot)}(\Omega)$.
    Owing to embeddings \eqref{nov158} and \eqref{nov171}, the sequence $\{u_k\}$ is also bounded in $W^{1,1}(\Omega)$. Since the embedding $W^{1,1}(\Omega) \embedding L^1(\Omega)$ is compact, there exists a  subsequence, still denoted by  $\{u_k\}$ which converges in $L^1(\Omega)$ and a.e. in $\Omega$ to some function $u \in  L^1(\Omega)$. We claim that $u_k \to u$ also in $L^{\vartheta(\cdot)}(\Omega)$. This claim will follow if we show that, for every $\gamma >0$,
    \begin{align}
\label{eq:maincompact1}
        \limsup_{k\to \infty} \int_\Omega \vartheta\bigg(x,\frac{\abs{u_k-u}}{\gamma}\bigg)\,dx \le 1.
    \end{align}
    By Theorem~\ref{thm:main'}, there exists a constant $K$ such that $\sup_k \norm{u_k}_{L^{\phi_{n,\diamond}(\cdot)}(\Omega)} \leq K$.  Since  $u_k \to u$ a.e. in $\Omega$, \cite[Theorem~2.3.17]{DieHarHasRuz17} ensures that   $\norm{u}_{L^{\phi_{n,\diamond}(\cdot)}(\Omega)} \le K$ as well.
    The assumption that
  $\vartheta$ grows essentially more slowly than $\phi_{n,\diamond}$ near infinity entails that there exists $t_0>0$ such that
    \begin{align*}
        \vartheta\bigg(x,\frac{t}{\gamma}\bigg) \le \phi_{n,\diamond}\bigg(x,\frac{t}{2K}\bigg) \qquad \text{for  $t \ge t_0$ and a.e. $x\in \Omega$.}
    \end{align*}
    Hence,
\begin{align}\label{nov163}
      \int_\Omega\vartheta\bigg(x,\frac{\abs{u_k-u}}{\gamma}\bigg)\,dx
        &\le \int_{\set{\abs{u_k-u}\le t_0}} \vartheta\bigg(x,\frac{\abs{u_k-u}}{\gamma} \bigg) \,dx
        + \int_{\set{\abs{u_k-u}> t_0}} \vartheta\bigg(x,\frac{\abs{u_k-u}}{\gamma}\bigg) \,dx
        \\  \nonumber
        &\le  \int_{\set{\abs{u_k-u}\le t_0}} \vartheta\bigg(x,\frac{\abs{u_k-u}}{\gamma} \bigg) \,dx
        + \int_{\set{\abs{u_k-u}> t_0}} \phi_{n,\diamond}\bigg(x,\frac{\abs{u_k-u}}{2K}\bigg) \,dx
    \end{align} 
    for $k \in \mathbb N$.
    Inasmuch $\norm{u_k-u}_{\phi_{n,\diamond}(\dot)(\Omega)} \le 2K$, we have that
\begin{align}\label{nov164}
\int_{\set{\abs{u_k-u}> t_0}} \phi_{n,\diamond}\bigg(x,\frac{\abs{u_k-u}}{2K}\bigg) \,dx \leq 1
    \end{align}
    for $k \in \mathbb N$. On the other hand, $\abs{u_k-u} \to 0$ a.e. in $\Omega$ and 
    $$ \vartheta\bigg(x,\frac{\abs{u_k-u}}{\gamma} \bigg) \leq \vartheta(x,t_0) \quad \text{for a.e. $x\in \Omega$,}$$
    for $k \in \mathbb N$.
    Moreover, owing to assumption \eqref{nov170}, $\vartheta(\cdot,t_0) \in L^1(\Omega)$. Therefore, by the dominated convergence theorem, 
      \begin{align}\label{nov166}
  \lim_{k \to \infty} \int_{\set{\abs{u_k-u}\le t_0}} \vartheta\bigg(x,\frac{\abs{u_k-u}}{\gamma} \bigg) \,dx=0.
    \end{align}
    Equation \eqref{eq:maincompact1} follows from \eqref{nov163}--\eqref{nov166}.
\end{proof}

\bigskip
 \section*{Compliance with Ethical Standards}\label{conflicts}

\smallskip
\par\noindent 
{\bf Funding}. This research was partly funded by:   
\\ (i) Research Project 201758MTR2  of the Italian Ministry of Education, University and
Research (MIUR) Prin 2017 ``Direct and inverse problems for partial differential equations: theoretical aspects and applications'';   
\\ (ii) GNAMPA   of the Italian INdAM - National Institute of High Mathematics (grant number not available).
\\ (iii) Deutsche Forschungsge- meinschaft (DFG, German Research Foundation) – SFB 1283/2 2021 – 317210226.

\bigskip
\par\noindent
{\bf Conflict of Interest}. The authors declare that they have no conflict of interest.

\printbibliography

\end{document}